\documentclass{article}
\pdfoutput=1 
\usepackage{url}
\usepackage[english]{babel}
\usepackage{amsmath,amsthm}
\usepackage{amsfonts}
\usepackage{mathbbol}
\usepackage{multirow}
\usepackage{subfigure}

\newtheorem{thm}{Theorem}[section]

\newtheorem{alg}[thm]{Algorithm}
\newtheorem{prop}[thm]{Proposition}
\theoremstyle{definition}

\theoremstyle{remark}
\newtheorem{rem}[thm]{Remark}
\numberwithin{equation}{section}
\usepackage{graphicx}

\begin{document}
\title{  Recovering   rank-one matrices via rank-$r$ matrices relaxation}

\author{\normalsize   Pengwen~Chen\footnote{ Applied Mathematics, National Chung Hsing University, Taiwan}, Hung~Hung\footnote{
 Institute of Epidemiology and Preventive Medicine, National Taiwan University, Taiwan}
}
\maketitle
\abstract{ 

PhaseLift, proposed by E.J. Cand\`{e}s et al., is one convex relaxation approach for phase retrieval. The relaxation enlarges the solution set  from rank one matrices to positive semidefinite matrices. 
 In this paper,   a  relaxation  is employed to  nonconvex alternating minimization methods to recover  the rank-one matrices.
 A generic measurement matrix can be standardized to a matrix consisting of orthonormal columns. To recover the rank-one matrix,
the standardized  frames  are used to select   the matrix with  the maximal   leading eigenvalue among the rank-$r$ matrices.
Empirical studies are conducted to validate the effectiveness of this relaxation approach. 
In the case of Gaussian random matrices  with  a sufficient number  of nearly orthogonal sensing vectors, 
 we show that the singular vector corresponding to the least singular value is  close to the unknown signal, and thus it can be a good initialization for the nonconvex 
 minimization algorithm.

}

\section{Introduction}
Phase retrieval is one important  inverse problem that arises in various fields, including  electron microscopy,  crystallography,  astronomy, and  optics. \cite{Millane:90, hurt2001phase, MiaoCell, Review,Fannjiang:12, IOPORT.06071995}. Phase retrieval aims to  recover signals from magnitude measurements only (optical devices do not allow direct recording of the phase of the electromagnetic field).

Let  $x_0\in \mathbf{R}^n$ or $x_0\in \mathbf{C}^n$ be some nonzero unknown vector to be measured.
Let $A\in \mathbf{R}^{N\times n}$ be the matrix whose   rows  are  sensing vectors $\{a_i\in\mathbf{R}^n\}_{i=1}^N$ or $\{ a_i\in \mathbf{C}^n\}_{i=1}^N$. The measurement  vector $b\in \mathbf{R}^N$ is the   magnitude, 
\[ \textrm{ $b=|Ax_0|$, or $ b_i=| {a}_i\cdot x_0|$ for $i=1,\ldots, N$. }\] Obviously, the signal $x_0$ can be determined up to a global phase factor at best, i.e., becasue  \[ |x_0\cdot a_j e^{i\theta}|=|x_0\cdot a_j| \textrm{ for any $\theta\in [0,2\pi]$},\] then $x_0 e^{i\theta}$ is also a solution.  The  recovery of $x_0 e^{i\theta}$ is referred to as the exact recovery.
When $A$ is a Fourier matrix, the problem is known as phase retrieval.
With this specific measurement matrix,  the task becomes more demanding, because Fourier magnitude is not only preserved under  global phase shift, but also  under spatial shift and conjugate inversion, which yields   twin images\cite{IOPORT.06071995}.

The first widely accepted phase retrieval algorithm was presented by Gerchberg and Saxton\cite{GS}.  Fienup\cite{Fienup:82} developed  the convergence analysis of the error-reduction algorithm and proposed   input-output iterative algorithms. The  basic and hybrid input-output algorithms can be viewed as a nonconvex Dykstra algorithm and a nonconvex Douglas-Rachford algorithm, respectively\cite{Bauschke02phaseretrieval}.   Empirically,  the hybrid input-output  algorithm is observed to   converge to
a global minimum (no theoretical proof is available)\cite{Review}.

The  major obstacle to  phase retrieval  is caused by the lack of convexity of the magnitude  constraint\cite{IOPORT.06071995}. 
  PhaseLift\cite{CPAM}, proposed by E.J. Cand\`{e}s et al., is one convex relaxation approach for phase retrieval. The relaxation changes the problem of vector recovery into a rank-one matrix recovery.   The global optimal solution can  be achieved, when  $A$ is a Gaussian random matrix and $N\ge Cn$ with some absolutely constant $C$\cite{FCM}.  To some extent,  this approach provides  a solution  to the phase retrieval problem, at least from the theoretical perspective, provided that the feasible set can  shrink to one single point under a sufficient number of measurements. In practice, the sensing matrix $A$ does not belong to this specific Gaussian model or  uniform models,  and   the  computational load of solving the  convex feasibility problem  can be   too demanding. In particular, it  requires the  computation of all the singular values in each iteration.

 In this paper, we explore the possibility  of using the rank-$r$ matrix relaxation in  phase retrieval.    In the first section, to illustrate the idea,  we review the exact recovery condition in  PhaseLift. Typically, the exact recovery of rank-one matrices  requires  a large   $N/n$ ratio. We \textit{standardize}  the   frame, such that  each matrix in the feasible set  has an  equal trace norm. Then,  the desired rank one matrix is the matrix whose leading eigenvalue is maximized. Gradually  enlarging the leading eigenvalue, the matrix  moves towards the   rank one matrix with high probability. Our  simulation result  substantiates the effectiveness of  recovering rank one matrices. 

To reduce the computational load, in section 2,   we apply  the relaxation to the   nonconvex alternating direction minimization method (ADM) proposed in~\cite{Wen}.
 Frames are standardized  to ensure the   equal trace  among  all  feasible solutions.  In theory, searching for  the optimal solution in a higher dimensional space can alleviate the stagnation of   local optima.  Finally,    with  a sufficient amount of nearly orthogonal sensing vectors,  we show that the corresponding singular vector is  close to the unknown signal and can thus  be a good initialization. To some extent,  this theoretical result provides a partial  answer to the solvability of phase retrieval.   In fact, when there is a lack of nearly orthogonal sensing vectors, the ADM can fail to converge, as discussed in  Section~\ref{Failure1}. 

 In section 3, we conduct a few experiments to demonstrate the performance of the ADM methods, including the convergence failure of nonconvex ADM, the comparison between rank one ADM to rank-$r$ ADM, and the application of phase retrieval computer simulations. Finally, given a generic matrix, we can find an equivalent matrix  whose columns are  orthogonal and  whose rows have 
 equal norm. We discuss  the existence and uniqueness  proof of the orthogonal factorization   in the appendix.
  
\subsection{Notation}
In this paper, we use the following notations. Let  $x^\top$ be the Hermitian conjugate of $x$, where $x$
 can be real or complex matrices (or vectors).  Hence, $x$ is Hermitian if $x=x^\top$. The notation $x^*$ is reserved
  for  a limit point of a  sequence $\{x^k\}_{k=0}^\infty$ or the final iteration of $x$ in the computation. Let $\|x\|_F$ be the Frobenius norm.
The function $diag(X)$ produces a vector that   is the diagonal of a matrix $X$.  The pseudo-inverse of   matrix $X$ is denoted by  $X^\dagger$.
The vector $e$ is a vector consisting of 
 one, and $e_j$ is the vector consisting of zero, except one at  the $j^{th}$ entry. 
 Let $x_0\in \mathbf{R}^n$ be the unknown signal and $A\in \mathbf{R}^{N\times n}$ or $\in \mathbf{C}^{N\times n}$ be the sensing matrix. Hence $N$ is the number of measurements.

\subsection{Ratio $N/n$}

We shall briefly outline   the threshold  ratio $N/n $ on the exact recovery of $x_0$~\cite{balan2006signal}. The result can be regarded as a worst-case bound, because we demand the exact recovery for all possible nonzero vectors $x$.
Denote a nonlinear map associated with $A$ by $M^A: \mathbf{R}^n\to \mathbf{R}^N$,
\[
M^A(x)=\sum_{k=1}^N |a_k\cdot x| e_k.
\]
The range of the mapping $M^A$ consists of all the possible measurement vectors $b$ via the sensing matrix $A$.

Throughout this paper, we assume that $A$ has rank $n$.
  We say that a matrix $A\in \mathbf{R}^{N\times n}$ satisfies the \textit{rank* condition} if 
 all square  $n$-by-$n$ sub-matrices of $A$ has  full rank  and $N> n$.
That is,  any $n$ row vectors of  $A$ are  linearly independent.
 
\begin{prop}\label{2n} Suppose  that $A$ satisfies the  rank* condition.  If $N\ge 2n-1$, then $M^A: \mathbf{R}^{n}\to \mathbf{R}^N$ is injective.
\end{prop}
\begin{proof} Suppose that $M^A (x)=M^A(\hat x)$ with $x\neq \hat x$; then $|a_k\cdot x|=|a_k\cdot \hat x|$. Rearrange the indices and assume \[ a_k\cdot x=a_k\cdot \hat x \textrm{ for } k=1,\ldots, l,  \]
\[ a_k\cdot x=-a_k\cdot \hat x \textrm{ for } k=l+1,\ldots, N. \]
 Because $N\ge 2n-1$, then either $l\ge n$ or $N-l\ge n$. Suppose $l\ge n$. Then 
$x-\hat x\in \mathbf{R}^n$ is orthogonal to  $ a_1,\ldots, a_l$. 
The full rank condition  yields $x-\hat x=0$, which shows the 
nonexistence of two distinct vectors $x,\hat x$.  Similar arguments apply to the case $N-l\ge n$.
\end{proof}
According to the above  proof, when $N\le 2n-2$, we can find a pair of vectors $x,\hat x$ such that $|Ax|=b=|A\hat x|$. Indeed, when $N=2n-2$, let $u$ be the vector orthogonal to $\{a_i\}_{i=1}^{n-1}$ and $v$  be the vector orthogonal to $\{a_i\}_{i=n}^{2n-2}$. Then  $x=u+v$ and $\hat x=u-v$ are the desired pair of vectors.  
However, for any particular vector $x$, it is possible that no $\hat x\in \mathbf{R}^n$ exists in the case $n+1\le N\le 2n-2$. 
\begin{prop} Fix $x_0\in \mathbf{R}^n$. Suppose each row $a_i$ of $A$ is independently sampled from some continuous distribution on the  unit sphere in $\mathbf{R}^n$. Let $b=|Ax_0|$. Then, with probability one, $|Ax|=b$ has a unique solution $x=x_0$ for $N\ge n+1$.
\end{prop}
\begin{proof}  
Assume $N=n+1$.
Write  \[A:=
\left[
\begin{array}{c}
A_1  \\
  A_2
\end{array}
\right] \textrm{ with } A_1\in\mathbf{R}^{n,n}, A_2\in \mathbf{R}^{1,n}.\] Then with probability one, $A_1$ is full rank and thus we can find a unique nonzero vector $c\in\mathbf{R}^n$ such that $a_{n+1}=c^\top A_1$. Clearly, $c$ is a continuous random vector that depends on $A_2$. 

Suppose that $x=\hat x$ is another solution of  $|Ax|=b$.  Then $\hat x$ should be one solution of  $2^n$ possible systems \[ \{ (A_1 \hat x)_i=\pm b_i \} \textrm{ for $i=1,\ldots, n$.}\] Let  $y:=A_1(x_0\pm \hat x)$. Then, $y$  must be one of the $3^n$ vectors with $y_i=\pm 2b_i $ or $0$ for $i=1,\ldots, n$.  Note that $y$ is independent of the selection of $A_2$.   Alternatively,  $a_{n+1} x_0=\mp a_{n+1}\hat x$ yields the orthogonality between $c$ and  $ A_1 (x_0\pm \hat x)$, i.e., \[ a_{n+1}\cdot (x_0\pm \hat x)=c^\top A_1  (x_0\pm \hat x)=c\cdot y=0.\]
Since  $c$ is a continuous random vector  that depends on $A_2$,  then with probability one, $c\cdot y=0$ leads to  $y=0$, which  implies that $x=\pm \hat x$ ($A_1$ is full rank).      
\end{proof}

However,  for generic  complex frames, the map  is  injective if  
 $N\ge 4n-2$, i.e., all vectors $x_0\in \mathbf{C}^n$ can be recovered. 
To recover a fixed vector $x_0$,  $N\ge 2n$ is a necessary condition. Interested readers are referred to the discussion in~\cite{balan2006signal} and \cite{Bandeira2013}.

One naive thought is that as $N$ grows faster than the speed of $n$, the rank one matrix can  be recovered. Unfortunately, this can be incorrect in some circumstances. 
We can construct  some matrix  $A\in\mathbf{R}^{N\times n}$ with $N$ being  order of $2^n$, but  some vector  $x_0 $ still cannot be recovered due to the failure of the rank* condition,  see the following remark.
\begin{rem}(Bernoulli random matrices)
We construct an example, in which   $x_0=e_1$ cannot be recovered from the measurement $|Ax_0|=b=e$.
Denote by $S\subset \mathbf{R}^{n}$ a set of vectors whose entries are $\pm 1$. There are $2^n$ vectors in $S$. Pick any subset of $N$ vectors from $S$ as $\{a_i\}_{i=1}^N$ ($A\in\mathbf{R}^{N\times n}$ is the Bernoulli random matrix). All the vectors $e_j, j=1,\ldots, n$ satisfy  $|Ae_j|=|Ax_0|=b=e$. Since  these matrices are indistinguishable,  $M^A$ does not  the injective property. 
 Note that the rank of the random matrix $A$ is $n$ in most cases. The rank $n$ condition on $A$, together with a large $N$,  does not imply the exact recovery of $x_0$. 
One can easily verify that  the Fourier matrix yields  the same difficulty. 

\end{rem}

 \subsection{PhaseLift}
Next, we introduce the PhaseLift method proposed by Cand\`{e}s et al.\cite{CPAM}.
 To simplify the discussion, we focus on the noiseless case.
 Introduce the linear operator on Hermitian matrices, 
 \[
 \mathcal{A}: \mathcal{H}^{n\times n}\to \mathcal{R}^N,\; \mathcal{A}(X):=diag(A^\top X A)=b^2, \; b_i^2=a_i^\top Xa_i, i=1,\ldots, N.
 \]
 An equivalent condition of $|Ax_0|=b$ is that $X:=x_0 x_0^\top $ is a rank-one solution to $\mathcal{A}(X)=b^2$. 
Hence,  the phase retrieval problem can be formulated as the matrix recovery problem,
 \[
 \min_{X} rank(X) \textrm{ subject to } \mathcal{A}(X)=b^2, X\succeq 0.
 \] 
 By factorizing a rank one solution of $X$, we can recover the  signal $x_0$.  

 
 To overcome the difficulty of rank minimization,   Cand\`{e}s et al.~\cite{CPAM} propose a convex relaxation of the rank minimization problem, which  is the trace minimization problem, 
 \[
 \min_{X}  tr(X), \textrm{ subject to } \mathcal{A}(X)=b^2, X\succeq 0.
 \]
When \[ I_{n\times n}\in span\{a_i a_i^\top \}_{i=1}^n,\] the condition $\mathcal{A}(X)_i= tr(a_i a_i^\top  X)$ automatically  determines the trace  $ tr(X)$ of $X$ and then 
 the trace minimization objective is redundant. 
Recovering $X_0=x_0x_0^\top $ can be achieved via solving 
the following  convex  feasibility problem,
\begin{equation}\label{PhaseLift}
\{X: X\succeq 0,  \mathcal{A}(X)=b^2\}.
\end{equation}
In the next subsection, we will show that we can always remove the trace minimization objective via an orthogonal decomposition on $A$, either SVD or QR factorizations. 

 The following Prop.~\cite{CPAM} illustrates the optimality of   the feasibility problem, which is a key tool for justifying the exact recovery theoretically. The proof can be found in~\cite{demanet2012stable}.
 \begin{prop}
 Suppose that  the restriction of $\mathcal{A}$ to the tangent space at $X_0:=x_0x_0^\top $ is injective. One sufficient condition for the exact recovery is  the existence of $y\in \mathbf{R}^N$, such that \[ Y:=\mathcal{A}^\top y=A^\top diag(y) A=\sum_{i=1}^N y_i a_i a_i^\top \] satisfies \[
Y_T=0 \textrm{ and } Y_{T^\bot}\succ 0.
\]
\end{prop}

 The  proposition  states one sufficient condition under which $x_0x_0^\top $ can be recovered from the frame $A$. In the real case, 
 when $N\ge 2n-1$, the rank* condition on $A$  is one sufficient condition  to ensure the injective property of 
 the restriction of $\mathcal{A}$. 
  Indeed, for any $x_0\neq 0$, $Ax_0$ consists of at most $n-1$ zeros thanks to the rank* condition, thus $Ax_0$ consists of  at least  $n$ nonzero entries.
Since  the tangent space at $x_0x_0^\top $ consists of $\hat X$ in a form $x_0 x^\top +x x_0^\top $ with some  $x\in \mathbf{R}^n$,    then   \[  \mathcal{A}(\hat X)=\mathcal{A}(x_0 x^\top +x x_0^\top )=2(Ax_0)  (Ax)=0\in \mathbf{R}^N \]  yields   $x=0$ (due to  the rank*  condition), which implies   $\hat X=0$.



\subsection{Special frames}
We shall highlight  three special frames where  the feasible set only consists of one single point. Thus  the unknown signal $x_0$ can be recovered via PhaseLift.
In the first case, we show that a frame with
 $N=n+1$ measurement vectors are sufficient to determine the unknown matrix $X=x_0x_0^\top$.
\begin{prop} Suppose that $a_j=e_j$ for $j=1,\ldots, n$  and $a_{n+1}(j) x_0(j)> 0$ for $j=1,\ldots, n$. \footnote{ This condition states that the entries of $a_{n+1}$ have the same sign as the ones of $x_0$.} Then,  the feasible set of  PhaseLift consists of only one single point, $x_0x_0^\top$.
\end{prop}
\begin{proof}
From $a_j=e_j$,  we have \[ a_j^\top  X a_j=X_{j,j}=|x_0(j)|^2.\] The positive semidefinite  requirement of $X$ yields $X_{i,j}^2\le X_{i,i}X_{j,j}$. 
The measurement $a_{n+1}^\top X a_{n+1}=|a_{n+1}\cdot x_0|^2$ enforces $a_{n+1}^\top X a_{n+1}$ to reach its upper bound among $X$ being positive semidefinite, i.e.,   the inequalities in  the following relation  become equalities,  \[
a_{n+1}^\top Xa_{n+1}=\sum_{i,j} a_{n+1}(i) X_{i,j}a_{n+1}(j)\]\[\le \sum_{i,j} |a_{n+1}(i) a_{n+1}(j)| \sqrt{X_{i,i}}\sqrt{X_{j,j}}\le (\sum_{j} |a_{n+1}(j)|\sqrt{X_{j,j}})^2
=|a_{n+1}\cdot x_0|^2,\]
where we used the assumption $a_{n+1}(j) x_0(j)>0$ for all $j=1,\ldots, n$. Hence, $X_{i,j}^2=X_{i,i}X_{j,j}$ for all $i,j$, which implies $ X=xx^\top $
 is the only feasible point with $ x:=(a_{n+1}x_0/|a_{n+1}|)$.

\end{proof}

In the second case,  
the exact recovery is  obtained via  a set of sensing vectors orthogonal to $x_0$. 
\begin{prop} 
Suppose that    some $n-1$ linear independent sensing vectors among $\{a_i\}_{i=1}^n$ exist  such that $x_0\cdot a_i=0$;  then, PhaseLift with measurement matrix $A$ recovers the matrix $x_0x_0^\top$ exactly. 
\end{prop}
\begin{proof}
Without loss of generality, assume that  $x_0=e_1$ and write $A$ as  an $n\times n$  matrix,  
\[
A=
\left(
\begin{array}{cc}
1_{1\times 1} & *_{1\times (n-1)}\\
0_{(n-1)\times 1} &  A_1   
\end{array}
\right)
\]
where
 $A_1\in\mathbf{R}^{(n-1)\times (n-1)}$ has rank $n-1$, i.e., it consists of linear independent columns.
Choose  $y$ to be a vector with $y_i>0$ for $i\ge 2$. Then  $Y_T=0$ and for any $z\in\mathbf{R}^{n-1}$,
 \[
(A_1 z)^\top diag([y_2,\ldots, y_n]) A_1 z=0. \]
Because $y_i>0$ for all $i=2,\ldots, n$, then 
$ A_1z=0, \; i.e., z=0. $ Hence, $Y_{T^\bot}\succ 0$.
\end{proof}

This special choice of the first column of  $A$ indicates \textit{ the orthogonality between  $n-1$ sensing vectors $a_i$ and $x_0$}. 
However,  the orthogonality is  generally not  satisfied for arbitrary vector $x_0$.

In the third case,  the exact recovery can be obtained via some structured sensing matrix, which in fact fails the rank* condition ($M^A$ is not injective). 
\begin{prop} 
Suppose that $N=2n-1$ and the sensing vectors in $A$ are $a_i=e_i$ for $i=1,\ldots, n$ and \[ \textrm{ $a_{n+i}= e_i+\beta_i e_{i+1}$ with $\beta_i\neq 0$ for $i=1,\ldots, n-1$.}\] Suppose that the  entries of $x_0$ are nonzero.  Then PhaseLift with measurement matrix $A$ recovers the matrix $x_0x_0^\top$ exactly. 
\end{prop}
\begin{proof}
To simplify the discussion,  assume $x_0=e$ and replace vectors $a_i$ with vectors $a_i x_0$ for all $i$. 
Any matrix $X$ in the feasible set has the form,
\[
X\in \mathbf{R}^{n,n}=
\left(
\begin{array}{ccccc}
1& 1&  &&\\
1& 1& 1 &&\\
& 1& 1 &1&\\
& &  &\ldots&\\
& &  &1&1
\end{array}
\right),
\]
i.e.,  $X_{i,i+1}=X_{i,i}=X_{i+1,i}=1$.

Claim: Because $X$ is positive semidefinite,  any principal sub-matrices of $X$ are positive semidefinite, which implies that  $X=ee^\top$.\\
Start with  $\alpha:=X_{i_1,j_1}=X_{j_1,i_1}$ with $i_1=j_1+ 2$.   Consider the principal sub-matrix $\{X_{i,j}: i,j\in \{j_1,j_1+1,j_1+2\}\}$.
 Compute the determinant of this submatrix \[ -1+2\alpha-\alpha^2=-(1-\alpha)^2.\]  Hence, the nonnegative determinant yields   $\alpha=1$. Similar  arguments work for $i_1=j_1+3$,\ldots. In the end, all entries of $X$ must be $1$, i.e., $X=x_0x_0^\top$ is the only matrix in the feasible set.
\end{proof}

Readers can  apply the similar arguments  to the recovery of  $X_0=x_0x_0^\top$ with $x_0\in \mathbf{C}^n$ via the following matrix: Let  $A\in \mathbf{C}^{N,n}$ with $N=n+2(n-1)$ and 
\[
a_i=e_i,\; a_{n+i}= e_i+\beta_i e_{i+1},\textrm{ and  $a_{2n-1+i}= e_i+\gamma_i e_{i+1}$, where }
\]
\[
\textrm{   $\beta_i\neq \gamma_i $ are nonzero for $i=1,\ldots, n-1$}.
\]
See \cite{doi:10.1137/110848074} for more discussion on the usage of $N=3n-2$ sensing vectors.
%
%
%
%
%

\subsection{Reduction of $N/n$ via standardized frames}

%
%
%
%
%
%

In the following,  some orthogonality on $A$, $A^\top A=I_{n\times n}$ is expected  to implement the matrix recovery algorithm. We say that a measurement matrix (a frame)  $A$ is \textit{standardized} if $A$ consists of orthonormal  columns, i.e.,  $A^\top A=I_{n\times n}$.
In fact,
given any measurement matrix $A\in \mathbf{R}^{N\times n}$ with rank $n$, we can 
take the QR decomposition of the measurement matrix $A$, $A=QR$ with $Q\in \mathbf{R}^{N\times n}$ consisting of orthonormal  columns and $R\in \mathbf{R}^{n\times n}$ being upper triangular. 
The  rank of $A$ is equal to  the  rank of $R$. Hence, denoting $Rx$ by $y$,
the problem is reduced to  solving $y$ from the measurements \[
|Ax|=|QRx|=|Qy|=b.
\] 
Once $y$ is obtained, $x$ can be computed via simply inverting the matrix $R$. Hence, the original frame $A$ is \textit{equivalent} to the standardized frame $Q$ in the sense that the two transforms $A,Q$ have the same range. 
That is,  $M^A$ is injective if and only if $M^Q$ is injective.

In the section, we propose one modification on PhaseLift  to recover the rank one matrix $X_0$. 
The idea  is based on the following simple fact.
\textit{ Among the feasible set $\mathcal{A}(X)=b^2$ and $X$ being positive semidefinite, to recover the rank one solution, we should choose the matrix $X$ whose leading eigenvalue is maximized. }

 \begin{figure}[htbp]
\centering
\subfigure[One  optimal solution]{\includegraphics[width=4.8cm]{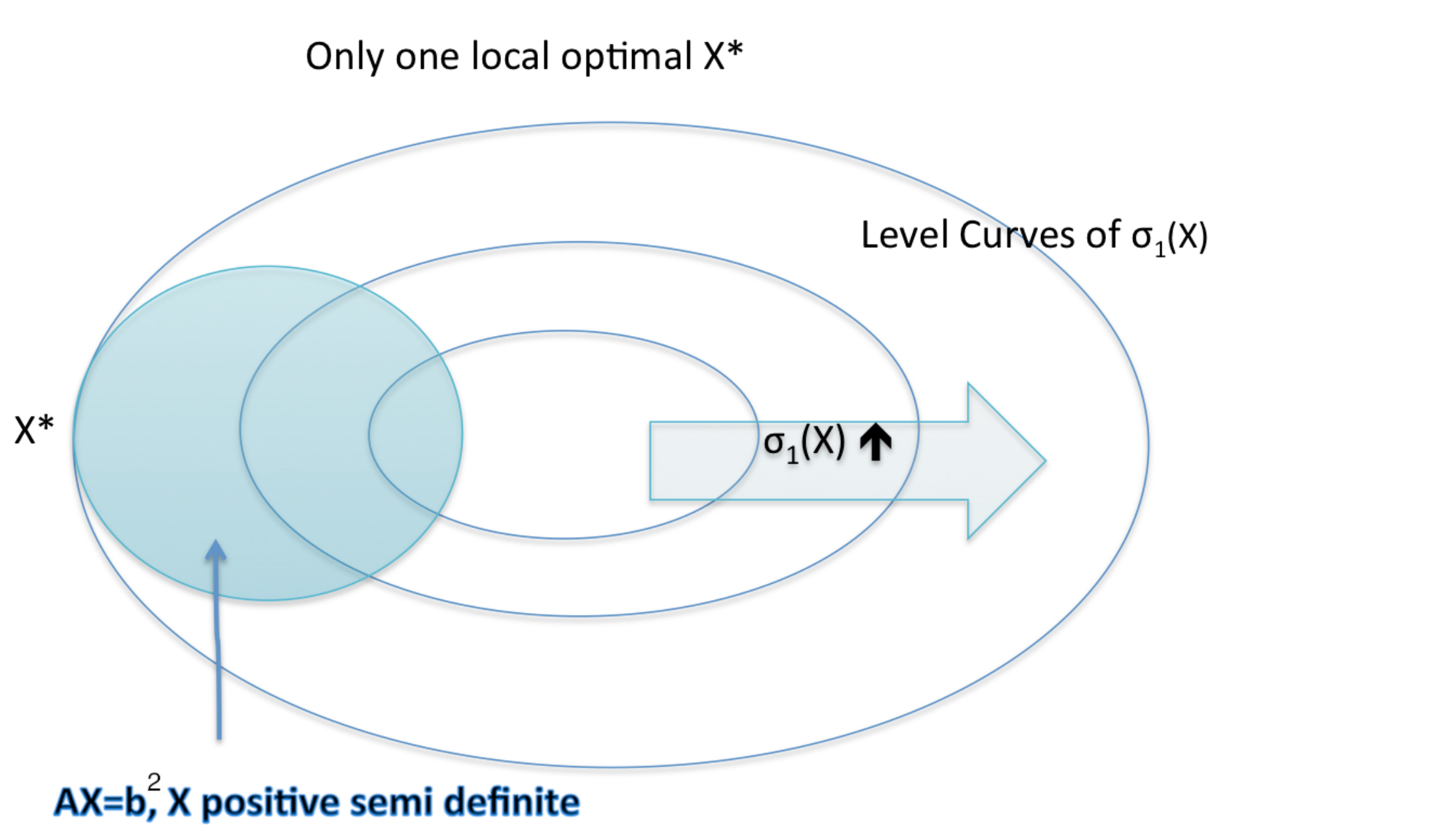}}~
\subfigure[Two optimal solutions]{\includegraphics[width=4.8cm]{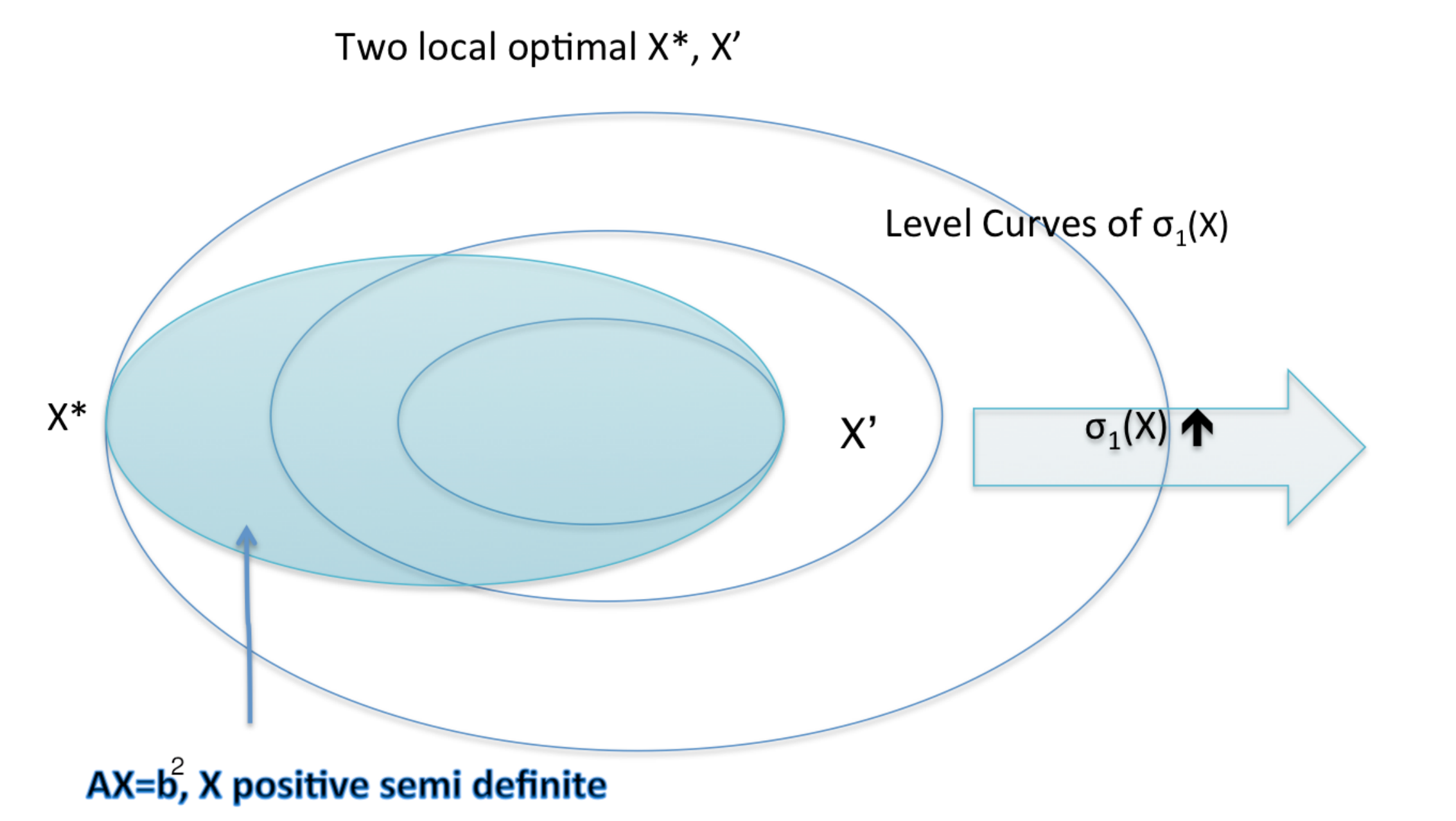}}
\caption{{  Maximizing the leading eigenvalue yields the rank one solution. }}
\label{local1}
\end{figure}

Consider the model
\[
\min_X - \sigma_1(X),
\]
subject to $X$ positive semidefinite and $\mathcal{A}X=b^2$,
where 
 $\sigma_1(X)$ refers to  the largest eigenvalue function of $X$.  See Fig.~\ref{local1}. 
 
 Then we have the following theoretical result. 
\begin{thm} Suppose that $A$ is a standardized matrix   with $N\ge 2n-1$ in the real case and $N\ge 4n-2$ in the complex case. Then with probability one, the global minimum occurs if and only if the minimizer $X$ is exactly $X=x_0 x_0^\top$.
\end{thm}

\begin{proof}
Because  $|Ax|^2=b^2$ and $A$ consists of  orthogonal columns, 
\[
e\cdot b^2=tr(X)=\sum_{i=1}^n \sigma_i(X), \]
 where $\sigma_i(X)$ refers to the $i$-th eigevlaue of $X$.
Because $X$ is positive semidefinite,  the largest eigenvalue of $X$, which cannot exceed $e\cdot b^2$, is maximized if and only if $X$ is a rank one matrix.   Finally, according to the above results, when $N$ exceeds the thresholds $2n-1$ or $4n-2$,  with probability one, the rank one matrix is unique, which completes the proof.  
\end{proof}

To address the problem, we propose  the following alternating direction method(ADM).
The ADM can be formulated as 
\[
\min_{X,Y} L_\beta(X,Y,\lambda):=-\sigma_1(X)-\lambda (X-Y)+\beta \|X-Y\|_F^2/2,
\]
subject to $X$ positive semidefinite and $\mathcal{A}Y=b$.
Hence, 
  the update of $X,Y$ is 
\begin{equation}\label{X}
arg\min_{X} -\sigma_1(X)+\beta \|X-Y-\lambda/\beta\|_F^2/2  \textrm{  subject to positive semidefinite,}
\end{equation}
\[
arg\min_Y  \beta \|X-Y-\lambda/\beta\|_F^2/2   \textrm{  subject to $\mathcal{A}Y=b^2$}.\]
The iteration becomes
\begin{enumerate}
\item Update $X$: Write $Y^k=UD_Y U^\top $;  then, 
thanks to  the rotational invariance of Frobenius norm, the minimizer in Eq.~(\ref{X}) is
\[
X^{k+1}=UD_X U^\top, \textrm{ where }\;   D_X=\max( D_{Y+\lambda/\beta},0)+\beta^{-1} e_1e_1^\top,
\]
where $D_{Y+\lambda/\beta}$ is the diagonal matrix of the eigenvalue decomposition of $Y^k+\lambda^k/\beta$ and the diagonal entries are in a decreasing order, i.e., the $(1,1)$ entry is the largest eigenvalue and will be added by $\beta^{-1}$ in the $X$ update.

\item Update $Y$ via the projection of $X-\lambda/\beta$,
\[
Y^{k+1}=\mathcal{A}^\top (\mathcal{A}\mathcal{A}^\top )^{-1} b+(I-\mathcal{A}^\top (\mathcal{A}\mathcal{A}^\top )^{-1} \mathcal{A})Z, \textrm{ where } Z=X^{k+1}-\lambda^k/\beta.
\]
The matrix $\mathcal{A}^\top (\mathcal{A}\mathcal{A}^\top )^{-1} \mathcal{A}$ is the orthogonal projector onto $Range(\mathcal{A}^\top )$ which is spanned by  $\{a_ia_i^\top \}_{i=1}^N$, each of which  has trace one.
\item Update $\lambda$:
\[
\lambda^{k+1}=\lambda^k-\beta (X^{k+1}-Y^{k+1}).
\]
\end{enumerate}

\begin{rem}[Counterexamples]

Because  $\sigma_1(X)$ is convex in $X$,  minimizing $-\sigma_1(X)$ yields  a non convex minimization problem.
Theoretically,  there is no guarantee that  the global optimal solution  can  always  be found numerically; however, the
 empirical study shows that the exact recovery will occur  with high probability.

Here is one counterexample. 
\[
A=\left(
\begin{array}{ccc}
  1 & 1  & 1  \\
  1 & -1  & -1  \\
  1 &  \sqrt{3/2} & 0  \\
  1 & -\sqrt{3/2}  & 0   \\
 1 & 0  &\sqrt{3}   \\
 1 & 0  &-\sqrt{3}   
\end{array}
\right)
\]
Then the feasible set consists
 of matrices  $\left(
\begin{array}{ccc}
  1-3\mu & 0  & 0  \\
  0 & 2\mu  & 0  \\
  0 &  0 & \mu   
\end{array}
\right)$ with $\mu\in [0,1/3)$. The maximization of the leading eigenvalue leads to 
   two possible solutions to $|Ax|^2=b^2$: one is $\mu=0$ (the rank-one solution) and the other is $\mu=1/3$ (the rank-two solution), 
depending on the initialization. See Fig.~\ref{localOpt0}.

\end{rem}

 \begin{figure}[htbp]
\begin{center}
\includegraphics[width=0.5\textwidth]{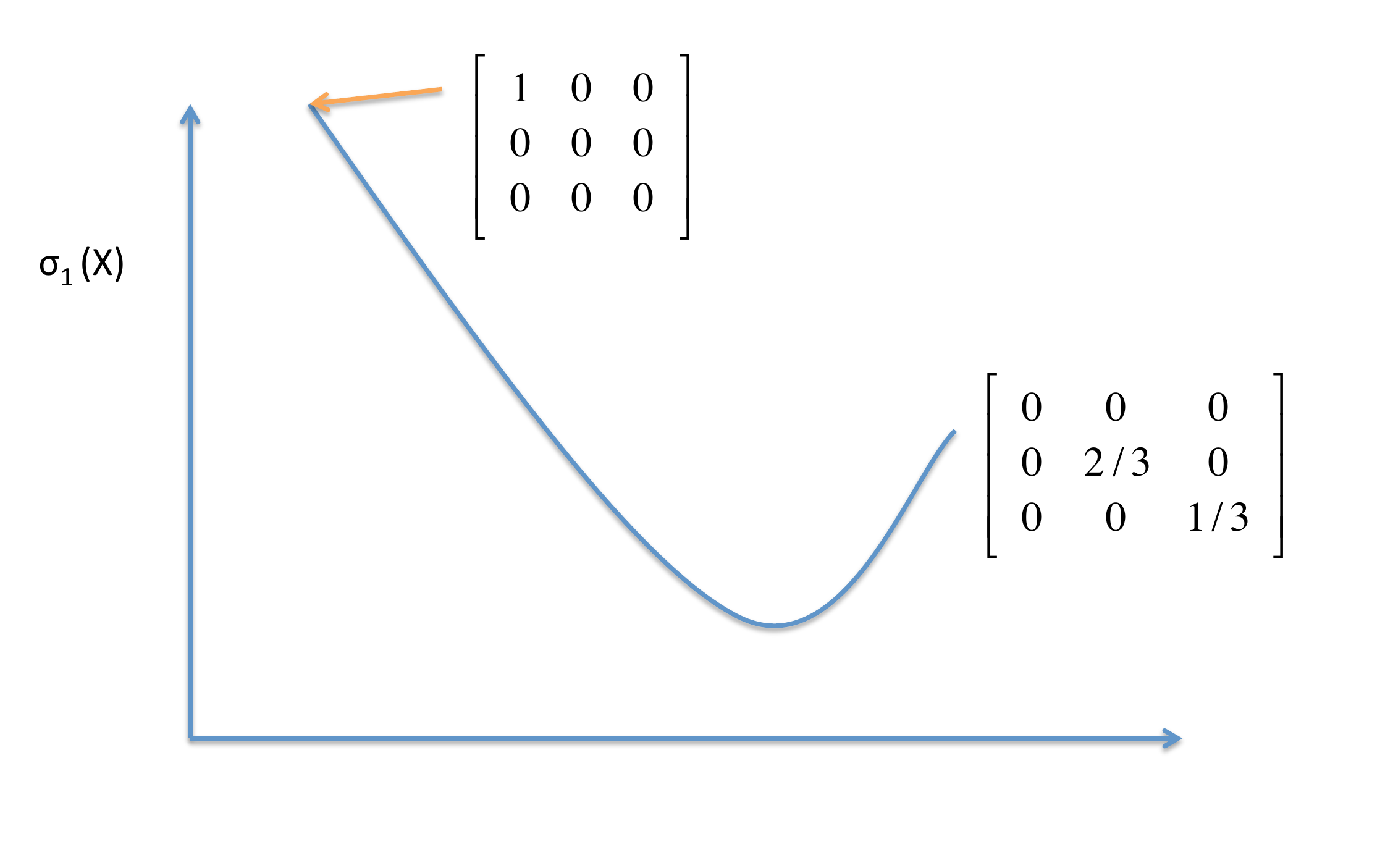}
\caption{{Maximizing the leading eigenvalue $\sigma_1(X)$ yields two local optimal solutions.}}
\label{localOpt0}
\end{center}
\end{figure}

 The following  experiments illustrate  that when the ratio  $N/n$ is not  large enough, the solution in PhaseLift is not rank one; we can successfully recover the rank one matrices via maximizing the leading eigenvalue;
 see Table~\ref{defaultR} for the real case and  Table~\ref{defaultC} for the complex case. 


 \begin{table}[htdp]
\caption{The number of  successes out of $50$ random  trials with $N=2n-1$. ``via Q" refers to the standardized measurement.}
\begin{center}

\begin{tabular}{|c|c|c|c|c|}
\hline
\multirow{2}{*}{n} &
\multicolumn{2}{c|}{PhaseLift} &
\multicolumn{2}{c|}{$\min-\sigma_1(X)$}\\
\cline{2-5}
  & via $A$ & via $Q$
  & via $A$ & via $Q$\\
\hline
 5&   37&    38&   13& 50 \\
  10&  28&   31&  11&49 \\
  15 & 17 &   21   & 13& 47 \\
   20 &16  &  25   & 15& 48 \\
   25 &10   & 13    & 4& 48\\
    30 & 3    & 7    &13& 48 \\
    35& 3     &4    &16& 47 \\
    40 & 1     &5  & 6& 46 \\
    45& 0     &0 &   4& 48\\
   50&  0     &0   & 12& 50 \\
\hline
\end{tabular}
\label{defaultR}
\end{center}
\end{table}

 \begin{table}[htdp]
\caption{The number of  successes out of $20$ random  trials via $N=2n-1, 3n-1, 4n-2$ standardized measurements ( complex case). }
\begin{center}
\begin{tabular}{|c|c|c|c|}
\hline
n& N=2n-1& N=3n-1 & N=4n-2\\ 
\hline
 5&   3&    20&    20\\
  10&  2&   20 &  20\\
  15 & 1 &   20   & 20\\
   20 & 0  &  20   & 20\\
   25 & 0   & 20   &20\\
    30 & 0    & 20    &20\\
    35& 0     & 20   &20\\
    40 & 0     &20  &  20\\
    45& 0     &20 &   20\\
   50&  0     &20   & 20\\
   \hline
\end{tabular}
\end{center}
\label{defaultC}
\end{table}%

The result of nonconvex minimization depends on the choice of initialization for $X^0$. When $X^0$ is near $X_0$, the exact recovery can be obtained. 
\begin{prop}   Let $X_0=x_0 x_0^\top $. Let $X$ be some positive semidefinite matrix.  
Let $f(t)$ be the spectral norm  of matrices $tX_0+(1-t)X$, \[
f(t)=\|tX_0+(1-t)X\|, \; 0\le t\le 1.
\]  
Let $v_1$ be the  unit eigenvector corresponding to  the largest eigenvalue  of $X$. If $tr(X_0 v_1v_1^\top )\ge \|X\|$, then $f(t)$ increases on the interval $[0,1]$. Hence, $\min (-\sigma_1(X))$ yields the recovery of $X_0$.
\end{prop}
\begin{proof}
Observe that $f(t)$ is convex in $t\in [0,1]$. It suffices to show that $f(t)\ge f(0)$ for  $t\in (0,1)$. According to the subdifferential of the matrix spectral norm\cite{Watson199233}, we have \[
t^{-1}(f(t)-f(0))\ge  tr((X_0-X)^\top G),
\] 
where $G$ is a subgradient of the spectral norm at $X$. Choose $G=v_1 v_1^\top $, then we have $ tr((X_0-X)^\top G)=tr(X_0 v_1v_1^\top )-f(0)$, which completes the proof. 
\end{proof}

\begin{rem} We provide a few examples to illustrate the recovery  of $X_0$ via $-\min \sigma_1(X)$.
Let $x_0=e\in \mathbf{R}^3$ and $X_0=ee^\top\in\mathbf{R}^{3\times 3}$. Suppose $a_1=e_1$, $a_2=e_2$ and $a_3=e_3$. Then, any feasible matrix has the form
\[
X=\left(
\begin{array}{ccc}
1  & 1-\alpha_1   & 1-\alpha_2   \\
 1-\alpha_1 & 1   & 1-\alpha_3   \\
1-\alpha_2  & 1-\alpha_3   & 1   
\end{array}
\right) \textrm{ with } \alpha_i\ge 0.
\]
Consider the case $\alpha_1=\alpha_2=\alpha$ and $\alpha_3=\alpha t$, then denote
\begin{equation} \label{M3}
X_{\alpha,t}:=\left(
\begin{array}{ccc}
1  & 1-\alpha   & 1-\alpha   \\
 1-\alpha & 1   & 1-\alpha t   \\
1-\alpha  & 1-\alpha t  & 1   
\end{array}
\right),
\end{equation}
and
 \[ \det(X_{\alpha,t})=\alpha^2 t(4-2\alpha -t).\] Thus, $X_{\alpha,t}$ is positive semidefinite if and only if  \[ \textrm{ $0\le t\le 4-2\alpha$ and $0\le \alpha\le 2$.}\]
 Hence, $X_{\alpha, 4-2\alpha}$ has two positive eigenvalues and one zero eigenvalue if $0<\alpha< 2$. For instance,  
when $\alpha=1/2$ and $t=3$, $X_{\alpha,t}$ has eigenvalues $1.5, 1.5, 0$. From the previous proposition, a positive semidefinite matrix  $X_{\alpha,t}$ with $0\le \alpha\le 1/2$ can return to $X_0$ via maximizing the leading eigenvalue.  
See Fig.~\ref{Matrix3}.
\end{rem}

 \begin{figure}[htbp]
\begin{center}
\includegraphics[width=0.3\textwidth]{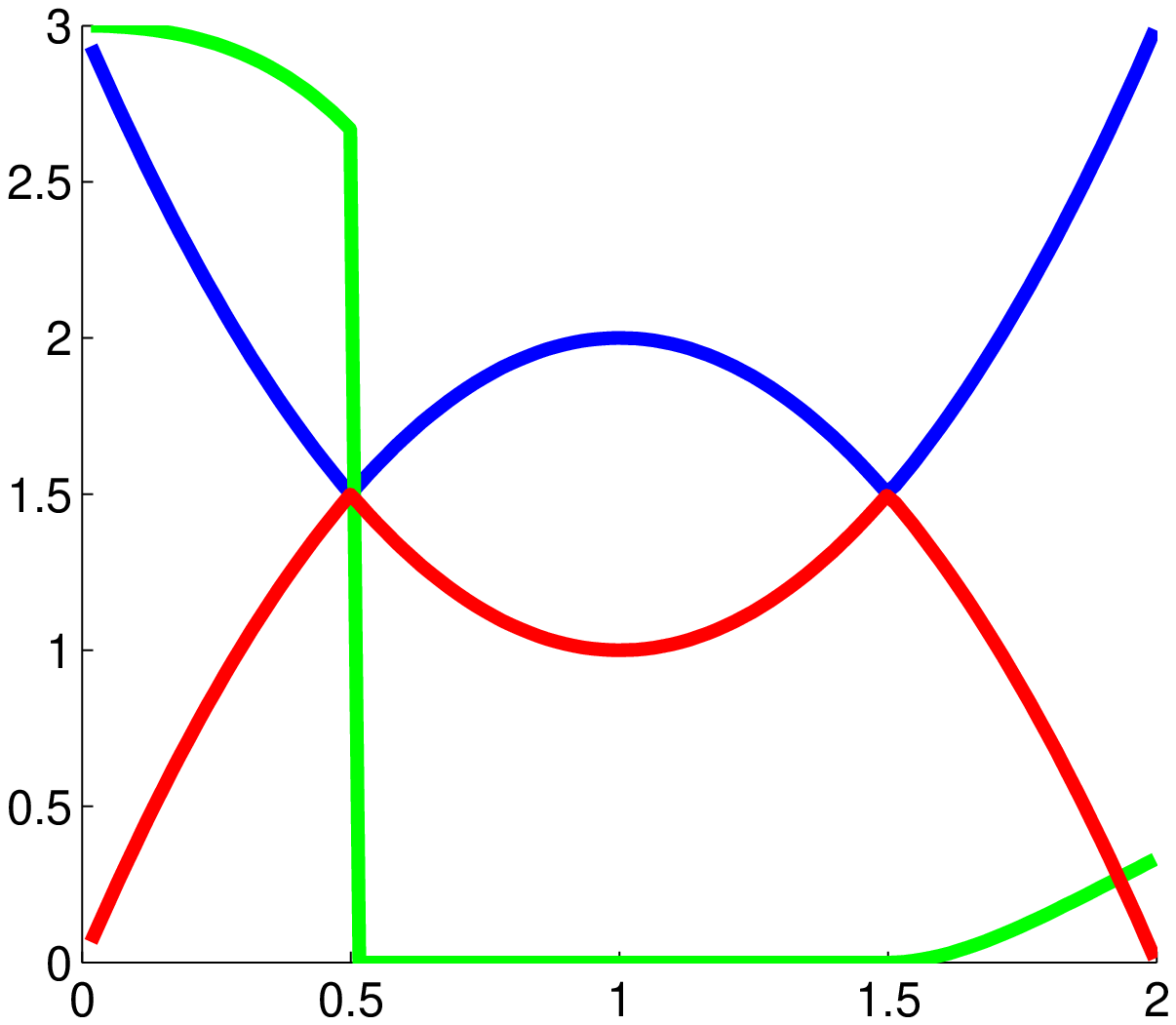}
\includegraphics[width=0.3\textwidth]{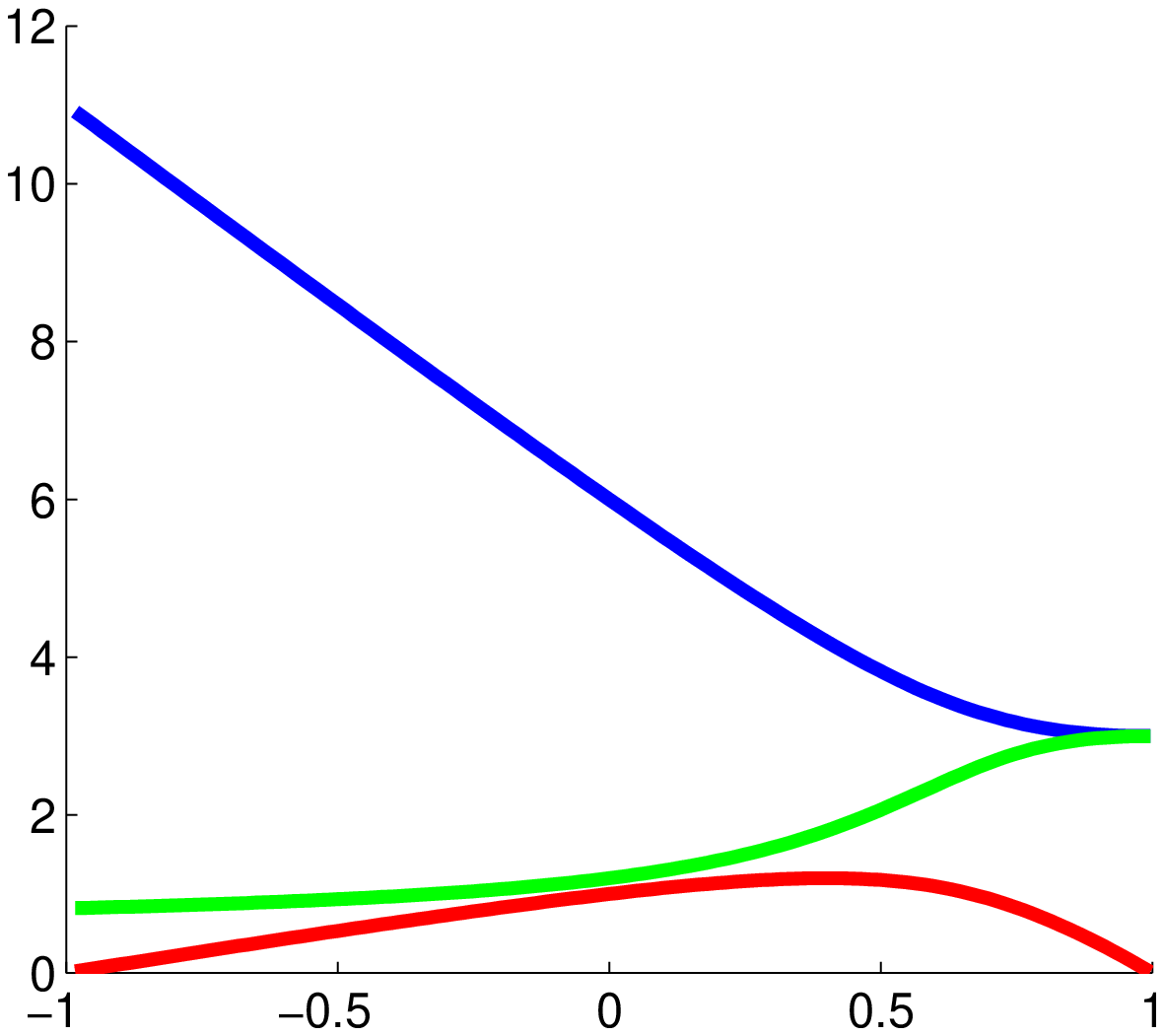}
\includegraphics[width=0.3\textwidth]{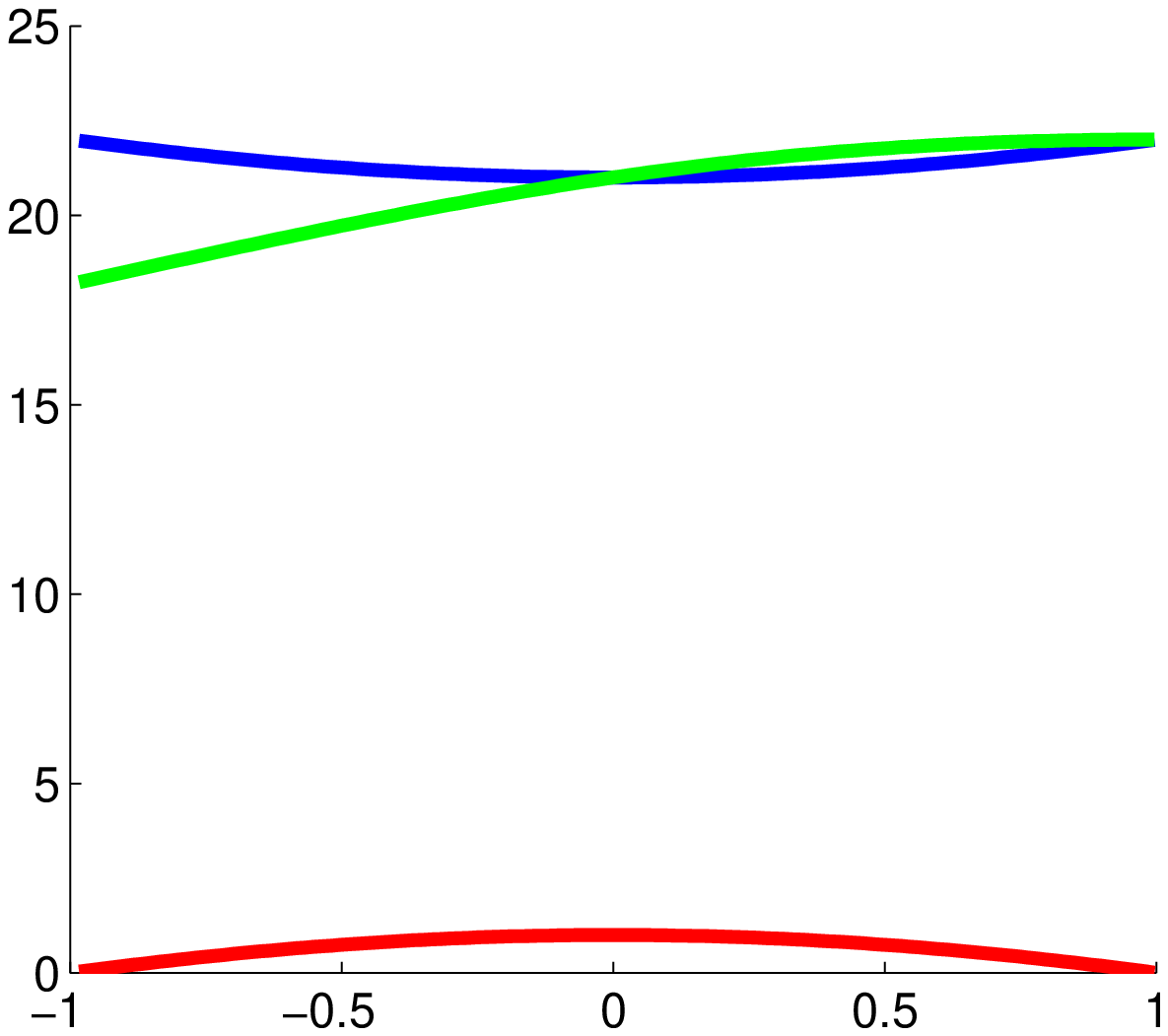}
\caption{{  The left subfigure shows $\sigma_1(X_{\alpha, 4-2\alpha})$ (blue), $\sigma_2(X_{\alpha, 4-2\alpha})$ (red) and $|x_0\cdot v_1|^2$ (green).
The middle and right subfigures show the results of $X_{\alpha,2\alpha-1}$ in Eq.~(\ref{QR2}) and $X_{\alpha}$ in Eq.~(\ref{QR3}), respectively.  The x-axis  represents  $\alpha$ values.
}}
\label{Matrix3}
\end{center}
\end{figure}

The next example illustrates the necessity of trace invariance  in recovering $X_0$. 
When the matrix  trace is not constant in the feasible set, then maximizing the leading eigenvalue does not recover $X_0$ in general. For instance, consider$X_0=x_0x_0^\top$ with
$x_0=e=[1, 1, 1]^\top$, and \[
A=\left(
\begin{array}{ccc}
 1 & 0  & 1  \\
 0 & 1  & 0  \\
 0 & 0  & 1  \\
1  & 1  & 2  
\end{array}
\right),\; b=|Ae|=\left(
\begin{array}{c}
 2 \\
 1  \\
 1  \\
4
\end{array}
\right).\]
Any matrix $X$ in the feasible set has the form 
\begin{equation}
X_{\alpha,\beta}:=\left(
\begin{array}{ccc}
3-2\beta  & \beta   & 2-\alpha   \\
\beta & 1   & \alpha   \\
2-\alpha  & \alpha   & 1   
\end{array}
\right) \textrm{ with } \det(X_{\alpha,\beta})=-(2\alpha-\beta-1)^2,
\end{equation}
and thus $ \det(X_{\alpha,\beta})\ge 0$ yields $\beta=2\alpha-1$. In fact, the feasible set consists of matrices \begin{equation}\label{QR2}
X_{\alpha,2\alpha-1}=(1-t) \hat e \hat e^\top+tee^\top, \hat e=[3,1,1]^\top, \; t=(\alpha+1)/2\in [0, 1].
\end{equation}
Maximizing the leading eigenvalue yields the solution $e\hat e$, which is not $X_0$. Alternatively, consider the QR factorization,
\[
A=QR=\left(
\begin{array}{ccc}
 1/\sqrt{2} & 0  & -1/\sqrt{3}  \\
 0 & 1  & 0  \\
 0 & 0  & 1/\sqrt{3}  \\
1/\sqrt{2}  & 0  & 1/\sqrt{3}  
\end{array}
\right)
\left(
\begin{array}{ccc}
 \sqrt{2} & \sqrt{2}  & \sqrt{2}  \\
 0 & 1  & 0  \\
 0 & 0  & \sqrt{3}  
\end{array}
\right),\] 
which yields
 the problem instead,
\[ b=|QRe|=|Qx_0|,\;  x_0:=Re=\left(
\begin{array}{c}
 3\sqrt{2} \\
 1  \\
\sqrt{3}
\end{array}
\right).\]
Then   the feasible set $
\{X: diag(QXQ^\top)=b^2, X\succeq 0\}
$ consists of 
\begin{equation}
X_{\alpha}:=\left(
\begin{array}{ccc}
18  & 3\sqrt{2}\alpha   & 3\sqrt{6}  \\
3\sqrt{2}\alpha & 1   & \sqrt{3}\alpha   \\
3\sqrt{6} & \sqrt{3}\alpha   & 3   
\end{array}
\right), \; \alpha\in\mathbf{R}.
\end{equation}
Let $x_0=[3\sqrt{2},-1,\sqrt{3}]^\top$ and $\hat x_0=[3\sqrt{2},-1,\sqrt{3}]^\top$.
 The feasible set consists of matrices \begin{equation}\label{QR3}
X_{\alpha}=(1-t) \hat x_0 \hat x_0^\top+tx_0x_0^\top, \; t=(\alpha+1)/2\in [0, 1].
\end{equation}
When $\alpha$ lies in $(0, 1)$, 
maximizing the leading eigenvalue of $X_{\alpha}$ yields the exact recovery of $X_0$.

%
%
%
%

\section{Low rank approaches}

In PhaseLift, 
all  eigenvalues of $X$ msy be  computed in each iteration  to be 
projected on the feasible set, consisting of  positive semidefinite matrices with rank $n$.  The projection obviously  becomes a laborious   task when $n$ is large.
Here we propose to replace  the feasible set with a subset consisting of rank-$r$ matrices,
 where $r$ is much smaller than $n$. 
 
Write the positive semidefinite matrices $X$ in PhaseLift as 
$X=xx^\top$ with $x\in \mathbf{R}^{n,r}$ or $x\in \mathbf{C}^{n,r}$. Then,  the original constraint in PhaseLift becomes
\[ b^2= \mathcal{A}(X)=diag(AXA^\top)=|Ax|^2.\]

\subsection{ADM with $r=1$}
Here we focus on the case $r=1$. In section~\ref{rankR}, we will discuss the case $r>1$.
When $r=1$, we arrive at   the problem,
\[
\textrm{ finding $x\in \mathbf{R}^n$ or $ \mathbf{C}^n$ satisfying } |Ax|= b.
\]
In~\cite{Wen}, the  framework
\begin{equation}\label{Wen}
\min \frac{1}{2}\||z|-b\|^2,\textrm{ subject to } z=Ax
\end{equation}
is proposed to address phase retrieval. 
They introduce
the augmented Lagrangian function
\begin{equation}
\label{WenL}
 L(z,x, \lambda)=\frac{1}{2}\||z|-b\|^2+\lambda\cdot (Ax-z)+\frac{\beta}{2}\|Ax-z\|^2.
\end{equation}
The algorithm consists of  updating $z,x,$ and $\lambda$  as follows.

 \begin{alg} \label{WenAlg}  
  Initialize $x^0$ randomly and $\hat \lambda^0=0$. Then repeat the steps for $k=0,1,2,\ldots$.
\begin{eqnarray*}
 z^{k+1}&=&\frac{u}{|u|} \frac{b+\beta |u|}{1+\beta},  \; u=Ax^k+\beta^{-1}\lambda^k,\\
 x^{k+1} &= & A^\dagger ( z^{k+1}-\beta^{-1} \lambda^{k}),\\
  \lambda^{k+1}&=& \lambda^k+\beta (A x^{k+1}-z^{k+1} ).
\end{eqnarray*}
 \end{alg}

Let us simplify the algorithm.   
Let  $P=AA^\dagger=QQ^\top$. Assume that $A$ has rank $n$.
By eliminating $x^{k}$, the  $\lambda$-iteration becomes  \[
\beta^{-1}\lambda^{k+1}=(I-P)(\beta^{-1} \lambda^k-z^{k+1}).
\]
Thus,   $A^\dagger \lambda^{k+1}=0$. In the end, we have the following algorithm.   

 \begin{alg}\label{WenAlg1} Denote $\hat \lambda^k = \beta^{-1}\lambda^k$. 
  Initialize $x^0$ randomly and $\hat \lambda^0=0$. Compute  $z^0$. Then, repeat the steps for $k=0,1,2,\ldots$,  
 \begin{eqnarray*}
 z^{k+1}&=&\frac{u}{|u|} \frac{b+\beta |u|}{1+\beta},  \; u=Pz^k+\hat \lambda^k,\\
 \hat \lambda^{k+1}&=&(I-P)(\hat \lambda^k-z^{k+1}).
 \end{eqnarray*}
 \end{alg}
 
\begin{rem}[Equivalence under right matrix multiplication]
  Note that the iteration is updated via $P=QQ^\top$, instead of $A$.  The matrix $R$ does not appear in the $z$ and $\lambda$ iterations.  Thus,   the algorithm is `` invariant" with respect to $R$.  That is, for   any invertible matrix $\hat R$, we get the same iterations $\{z^k,\hat \lambda^k\}_{k=1}^\infty$, when  $(A, x^0)$ is replaced by  \[ (Q\hat R, (\hat R)^{-1} Rx^0).\]  In particular,   the iteration with $Q$ yields the same result as the one with $A$ itself. However,  ADM can produce different results when the \textit{left} matrix multiplication on $A$ is considered. See section~\ref{Failure1}.
  \end{rem}

 Suppose that  $z^k$ converges  to  $z^*$ and $\hat \lambda^k$ converges to $\hat\lambda^*$. Then,  $Pz^*=z^*$ and $P\hat \lambda^*=0$.  Hence, 
 $x^*=A^\dagger z^*$,  and $z^*= Pz^*=A x^*$.  Consider  the limit of the $z$-iteration,
 \[
 (1+\beta) z^*=\frac{u}{|u|} (b+\beta |u|)=\frac{u}{|u|} b+\beta (z^*+\hat \lambda^*).
 \]
 Thus, we have the orthogonal projection of $bu/|u|$ onto the range of $A$ and its null space,  
  \[\frac{u}{|u|} b=z^*-\beta\hat \lambda^*,\] 
\begin{equation}\label{B}
\textrm{ thus } \|b\|^2=\| z^*\|^2+\beta^2 \|\hat \lambda^*\|^2.
 \end{equation}
 This result shows  $\|Ax^*\|\le \|b\|$. 
 In particular, when $A=Q$, we have  $\|x^*\|\le \|b\|=\|x_0\|$, i.e., any non-global solution has the smaller norm. 
Besides,   Eq.~(\ref{B}) suggests the usage of smaller $\beta$ to improve the recovery of $x_0$. Empirical experiments show that, starting with $\lambda^0=0$,  a smaller value $\beta$ leads to  a higher chance of exact recovery.

To analyze the convergence,  we write the function $ L(z,x, \lambda)$ as 
\[
 \hat L(z,x, \lambda, s):=\frac{1}{2}\|z-bs\|^2+\lambda\cdot (Ax-z)+\frac{\beta}{2}\|Ax-z\|^2,
\]
where the entries of $s$ satisfies $|s_i|=1$ for $i=1,\ldots, N$ and clearly  the optimal vector $s$ to minimize $\hat L$ is given by $u/|u|$.  When $s$ is fixed, then the following customized proximal point algorithm 
which consists of iterations
\begin{eqnarray*}
 z^{k+1}&=&s \frac{b+\beta |u|}{1+\beta},  \; u=Ax^k+\beta^{-1}\lambda^k,\\
  \lambda^{k+1}&=& \lambda^k+\beta (A x^{k}-z^{k+1} ),\\
   x^{k+1} &= & A^\dagger ( z^{k+1}-\beta^{-1} \lambda^{k+1})
\end{eqnarray*}
can be used to solve the least squares problem
\begin{equation}\label{LSE}
\min \frac{1}{2}\|z-sb\|^2,\textrm{ subject to } z=Ax.
\end{equation}
 Gu et al.~\cite{COAPHE} provide the convergence analysis of the customized proximal point algorithm. More precisely, fixing $s$,  let $(z^*,x^*,\lambda^*)$ be a saddle point of 
$\hat L(z,x,\lambda,s)$ and 
let \[ 
\|v^{k+1}-v^k\|_M^2:=(v^{k+1}-v^k)^\top M (v^{k+1}-v^k)
\textrm{ with }\] 
\[
M:=[\beta^{1/2}A, -\beta^{-1/2}I]^\top [\beta^{1/2}A, -\beta^{-1/2}I], \; v:=
\left(
\begin{array}{c}
x  \\
 \lambda
\end{array}
\right).
\]
In Lemma 4.2, Theorem 4.2 and Remark 7.1~\cite{COAPHE}, the sequence $\{v^k\}$ satisfies
\[
\|v^{k+1}-v^*\|_M^2+\|v^k-v^{k+1}\|_M^2\le \|v^k-v^*\|_M^2,
\]
and then  $\lim_{k\to \infty} \|v^k-v^{k+1}\|_M^2=0$. Any limit point of $[z^k,x^k,\lambda^k]$ is a solution of the problem in Eq.~(\ref{LSE}) with $s$ fixed.
However, 
 the convergence analysis  of the algorithm in Eq.~(\ref{WenAlg}) does not exist  due to the lack of convexity in $z$.
 In fact, when $s$ is updated in each $z$-iteration, this algorithm sometimes fails to converge, which is shown in our simulations; see Section~\ref{Failure1}.

\subsection{Recoverability}\label{Recovery}
   We make the following two observations regarding$|Ax_0|=b$. Suppose the unknown signal $x_0$ satisfies $\|x_0\|=1$.
First, the vector $x$ is updated to maximize the inner product $|Ax|\cdot b$ in the ADM.  However, because  the norm constraint $\|x\|=1$ is not enforced explicitly,  a non-global maximizer $x$ generally does not has the unit norm, $\|x\|<\|x_0\|=1$.   In fact, classic phase retrieval algorithms e.g., ER, BIO, HIO\cite{Fienup:82}, do not enforce the constraint directly.   Second,   for those indices $i$ with $a_i\cdot x$ close to zero, the unit  vector $x$ to be recovered should be approximately perpendicular to  these sensing vectors  $a_i$. The candidate set $\{x: |a_i\cdot x|\le b_i\}$  forms a cone, including  unit vectors approximately orthogonal to  $a_i$ corresponding to $b_i$ close to zero. 
In particular, when  $a_1\cdot x_0\neq 0$ and $a_i\cdot x_0=0$ for $i=2,\ldots, N$ with $N>n$, then the  cone is  exactly   the one-dimensional subspace spanned by   the vector $x_0$.

One important issue of non-convex minimization problems is that the initialization can affect the  performance  dramatically. The x-iteration in the ADM tends to produce a vector close to the singular vector corresponding to its least singular value of $A$, in the sense that   $A^\dagger z$ boosts the component  along   the singular vector corresponding to the largest singular value of  $A^\dagger$, i.e.,  the smallest singular value of  $A$. In the following, we will analyze the recovery problem from a viewpoint of singular vectors and derive an error estimate between the unknown signal and the singular vector. 
 
Rearrange the indices such that   $\{b_i\}$ are sorted in an increasing order, \[ 0\le b_1\le b_2\le\ldots\le b_N.\] 
Divide the indices into three groups,
\[
\{1,2,\ldots, N\}=I\cup II\cup III.
\]
We shall use subscripts $I,II,III$  to indicate the indices from these three groups.   The set $ I$ consists of the indices corresponding to  the smallest $N_{I}$  terms among $\{b_i\}$.  The set $  II  $ consists of the indices corresponding to  the largest $N_{II}$  terms among $\{b_i\}$. 
 Denote  the matrix   consisting of  rows $\{a_i\}_{i\in I} $ by $A_I$. Let  $A_I$ and $A_{II}$ consist of  $N_I$ and $N_{II}$ rows, respectively. 
In the following, we  illustrate that  the singular vector $x_{min}$ corresponding to the least singular value   of $A_I$ is a good initialization $x^0$ in the ADM.

Without loss of generality, assume  that $x_0=e_1$ and that all the rows $\{a_i\}_{i=1}^N$ of $A$ are normalized, $\|a_i\|=1$. Observe that the desired vector  satisfies  $|A x_0|\le b$ and $\|x_0\|=1$. Hence, we   look for a \textit{ unit}  vector $x$ in the closed convex set \[
|a_i\cdot x|\le b_i \textrm{  for all $i$}.
\]

Whether phase retrieval  can be solved depends on the structure of $A$. We make the following  assumptions.
\begin{itemize}
\item First, sufficiently many indices $i\in I$ exist, such that  
 \[ \|b_I\|^2:=\sum_{i\in I} b_i^2 \textrm{ is sufficiently small compared to } \|A_I x_1\|^2,\] 
where $x_1$ is a unit vector  orthogonal to $x_0$. (Clearly, the matrix $A_I$ has rank at least $n-1$.)

\item Second, there are at least $n$ indices  in $II$, such that  \[ \textrm{entries } \{b_i\} \textrm{ are \textit{large} for } i\in II , \]  and   the matrix $A_{II}$ has rank $n$.

\end{itemize}

The  assumption  that $\{b_i\}_{i\in I}$ is close to zero  implies that 
 $\{a_i\}_{i\in I}$ are almost orthogonal to $x_0$.
Thus, we instead  
  solve the problem \[ \textrm{ $\min_x \|A_I x\|^2$ with $\|x\|=1$.}\]  
  The minimizer denoted by $x_{min}$ is  the singular vector $x_{min}$ corresponding to the least singular value of $A_I$. Then, 
 \[
 \|A_I x_{min}\|\le \|b_I\|.
 \]
 Let $\{0\le \mu_1\le \ldots\le \mu_n\}_{i=1}^n$ be the singular values of $A_I$ with right singular vectors $v_i$.
 Then $v_1=\pm x_{min}$ and we can write \[
 x_0=\alpha_1 x_{min}+\sqrt{1-\alpha_1^2}\, w, \]
 with some unit vector $w$ orthogonal to $x_{min}$.
Let 
  \[ x_1:=-(1-\alpha_1^2)^{1/2} x_{min}+\alpha w,\] then $x_1$ is a unit vector  orthogonal to $x_0$. 
Note that   \[ \|x_0x_0^\top -x_{min}x_{min}^\top\|^2= \|x_0x_0^\top -x_{min}x_{min}^\top\|^2_F/2
=1-\alpha_1^2.\]
 The following proposition gives a bound for the distance $ \|x_0x_0^\top -x_{min}x_{min}^\top\|$ via the ratio $\|A_I x_0\|/\|A_I x_1\|$.

 \begin{prop} Let $\alpha_1=|x_0\cdot x_{min}|$. 
 Then,  
  \begin{eqnarray}\label{closeness}
 && (2-\alpha_1^2) \|b_I\|^2\ge (1-\alpha_1^2) \|A_Ix_1\|^2
  \end{eqnarray} 
 Therefore, as $\|b_I\|$ is small enough, $1-\alpha_1^2$ must be close to $0$.  
 Note that $ \|b_I\|^2=\|A_Ix_0\|^2$, then \begin{eqnarray}\label{closeness1}
 &&  \|b_I\|^2  \|A_Ix_1\|^{-2}\ge \left(\frac{  1-\alpha_1^2}{2-\alpha_1^2}\right) \ge
   \frac{1}{2}\|x_0x_0^\top -x_{min} x_{min}^\top \|^2.
  \end{eqnarray}

 \end{prop}
 \begin{proof} Note that
  \[
\begin{cases} \|A_Ix_0\|^2=\alpha_1^2\|A_Ix_{min}\|^2+(1-\alpha_1^2)\|A_Iw\|^2, \\
 \|A_Ix_1\|^2=(1-\alpha_1^2)\|A_Ix_{min}\|^2+\alpha_1^2\|A_Iw\|^2. \; 
  \end{cases}
  \] 
  
  Because  $x_{min}$ is the singular vector of the the least eigenvalue,
 \begin{eqnarray*}
&&(1-2\alpha_1^2) \|A_Ix_0\|^2-(1-\alpha_1^2)\|A_Ix_1\|^2\\
 &=&\|A_Ix_{min}\|^2+2(1-\alpha_1^2)^2 \left(
 \|A_Iw\|^2-\|A_Ix_{min}\|^2\right)\ge 0.
 \end{eqnarray*}
 \end{proof}

 Denote the sign vector of $A_{II} x_{min}$ by $u_{II}$,  \[
 u_{II}=\frac{ A_{II} x_{min}}{|A_{II}x_{min}|}. \]
The closeness $\|A_{II}(x_{min}-x_0)\|$ yields that  
 $A_{II}x_{min}$ should be close to $A_{II} x_0$. In particular, when the magnitude of entries $b_{II}$ are large enough, both vectors have the same sign  \[ \frac{ A_{II} x_{min}}{|A_{II}x_{min}|}=\frac{ A_{II} x_{0}}{|A_{II}x_{0}|}\] and then 
 \[ A_{II} x_{0}=u_{II}b_{II}.
 \]
 Once the sign vector is retrieved, the vector $x_0$ can be computed via
\[
x_0=A_{II}^{-1} (u_{II}b_{II}).
\]
 
\subsection{Real Gaussian matrices }
As an example of computing $\|b_I\|$ and $\|A_Ix_1\|$, 
let  $A\in \mathbf{R}^{N\times n}$  be  a  random matrix consisting of i.i.d. normal $(0,1)$ entries.  In the following, we would  illustrate that when $N_I/N$ is small enough,  $x_{min}$ is a good initialization for  $x_0$  and thus we can recover the missing sign
 vector $(Ax_0)/b$.

 Let $x_0=e_1$. Then $x:=a_i\cdot x_0$ follows the distribution 
  the normal $(0,1)$ distribution.
   Let $a>0$ be a function of $N_I/N$ and satisfy \begin{equation}\label{Feq}
 F(a):=\int_{-a}^a (2\pi)^{-1/2}\exp(-x^2/2) dx=N_I/N.\end{equation}  
 Then  the leading terms of Taylor series  of   Eq.~(\ref{Feq}) yields
 \begin{equation}\label{aF}
 (2/\pi)^{1/2}(a-a^3/6)\le  N_I/N\le (2/\pi)^{1/2} a.
 \end{equation}

Define the ``truncated" second moment
 \[
 \sigma_a^2:= \int_{-a}^a x^2 (2\pi)^{-1/2}\exp(-x^2/2)dx.
 \]
Taking the Taylor expansion of  $\sigma_a^2$ in terms of $a$ yields 
the following  result.
 \begin{prop}\label{sigmaB} 
\[  \sigma_a^2\le (2/\pi)^{1/2} a^3/3.\]
Additionally, when  $N_I/N$ tends to zero,  $\sigma^2_a$ is  approximately  $ (N_I/N)^3(\pi/6)$, where  Eq.~(\ref{aF}) is used.
   \end{prop}
  
When  $N_I/N$ tends to zero, the following proposition shows that  $N_I^{-1/2}\|b_I\|$ is approximately $ (\pi/6)^{1/2}N_I/N$,
 see Fig.~\ref{bnorm} for one numerical simulation. 
  
\begin{prop}\label{BnormA}
Suppose $N_I/N<c$ for some constant $c$. Then 
with high probability  \footnote{ An event occurs ``with high probability" if for any $\alpha\ge 1$, 
the probability is at least $1-c_\alpha N^{-\alpha}$, where $c_\alpha$ depends only on $\alpha$. }  \[
N_I^{-1/2}\|b_I\|\le c_6(\pi/6)^{1/2}N_I/N,
\]
for some constant $ c_6$. 
 
 \end{prop}
 \begin{proof} 

 Because $x_0=e_1$,
 $\|b_I\|$ is  the norm of the first column of $A_I$.
Let $\hat \mu>0$ be the sample $N_I$ quantile~\cite{ChenH}, $\hat \mu=|a_{N_I}\cdot x_0|$.
 We have the following probability inequality for $|\hat \mu-a|$:
 For every $\epsilon>0$, 
 \[
 \mathbb{P}(|\hat \mu-a|>\epsilon)\le 2\exp(-2N\delta_\epsilon^2)
\textrm{ for all $N$,}\] where \[\delta_\epsilon:=\min \{F(a+\epsilon)-N_I/N,N_I/N-F(a-\epsilon)  \}.\]
That is,  with high probability we have \begin{equation}\label{aa}
a-\epsilon\le |a_{N_I}\cdot x_0|\le a+\epsilon.
\end{equation}
The proof is based on  the Hoeffding inequality in large deviation;
see 
Theorem 7, p. 10, \cite{ChenH}.
Let $\hat N_I$ be  the cardinality  of the set 
$\hat I:=\{i:    |a_i\cdot x_0|\le  a\}$. With high probability, 
we have\footnote{ For  every $\epsilon_1>0$,
let  $\beta:=F^{-1} ( N_I(1-\epsilon_1)/N)$ and  $\epsilon:= a-\beta=F(N_I/N)-\beta>0$. Then   
\[ \beta-\epsilon\le |a_{ N_I/(1-\epsilon_1)}\cdot x_0|\le \beta+\epsilon \textrm{ with high probability.}\]
}   \[ \hat N_I \ge N_I (1-\epsilon_1)\textrm{ for any $\epsilon_1>0$}.\]

 Let $\{Z_i\}_{i=1}^N$ be independent  bounded random variables, \[ 
 Z_i :=
 (N/N_I)^3\left((a_i\cdot x_0)^2- \sigma_a^2(N/ N_I)\right) \textrm{ if } |a_i\cdot x_0|\le a, \textrm{ zero, otherwise.}
  \]
\[ \textrm{    Then } \mathbb{E}[Z_i]=  (N/N_I)^3\left(\sigma_a^2-\sigma_a^2 (N/ N_I) \mathbb{P}(|a_i\cdot x_0|\le a)\right)=0,\] i.e., 
   $Z_i$ is centered. The Hoeffding inequality 
   (e.g., Prop. 5.10 \cite{Roman})  
  yields for some positive constants $c_4,c_5$, \footnote{  The sub-gaussian norm $\|Z_i\|_{\psi_2}$is bounded by a constant ( depending on $N_I/N$), independent of $N$: \[\sup_{p\ge 1} p^{-1/2} (\mathbb{E}[Z_i^p])^{1/p}\le\sup_{p\ge 1} p^{-1/2}  (N/N_I)^2 a^2(N_I/N)^{1/p-1}\le (N/N_I)^2 a^2(N_I/N)^{-1}.\]}
  \[ \mathbb{P}(N^{-1}|\sum_{i=1}^N Z_i|\le  t)\ge c_4 \exp(- c_5 N t^2).\]
 That is, 
  with  probability at least $1- c_4 \exp(- c_5 N t^2)$,
 \[
\left | N^{-1} \|\hat b_{ I}\|^2- \sigma_a^2\frac{\hat N_I}{ N_I}\right|\ge t(N_I/ N)^3.
 \]
 
Thanks to  $\hat N_{ I}\ge N_I (1-\epsilon_1)$ with high probability and Eq.~ (\ref{aF}),~(\ref{aa}),  \[ 
\frac{\|b_I\|^2}{N_I}\le \frac{ \|\hat b_I\|^2}{\hat N_I}+
\frac{\epsilon_1 N_I (a+\epsilon)^2 }{\hat N_{ I}}\le  \frac{ \|\hat b_I\|^2}{\hat N_I}+
\epsilon_1 \frac{ (a+\epsilon)^2}{1-\epsilon_1},\]
   and
  \[ 
 \frac{\|\hat b_{ I}\|^2}{\hat N_{ I}}\le  \sigma_a^2 \frac{N}{ N_I}+t\frac{N_I}{\hat N_{ I}} \left(\frac{N_I}{N}\right)^2 \]\[\le 
\left( \sigma_a^2(N_I/N)^{-3}+t/(1-\epsilon_1) \right)
 ( \frac{N_I}{N})^2.
 \]
Together with Prop. \ref{sigmaB},  with high probability 
\[N_I^{-1/2}\|b_I\|\le c_6 (\pi/6)^{1/2} (N_I/N) \textrm{
for some constant $c_6$.} \]
  \end{proof}

To  bound the norm $\|x_0x_0^\top-x_{min}x_{min}^\top\|$ in 
 Eq.~(\ref{closeness1}), we need to compute $\|A_Ix_1\|^2$. 
  Denote by  $A'_{I}$  the sub-matrix of $A_I$ with the first column deleted, i.e., 
\[
A=[ *_{N_1\times 1}, A'_I].
\]
Denote by $\{\delta_{i}\}_{i=2}^{n}$ the singular values of the sub-matrix  $A_I'$.
Since $x_1$ is orthogonal to $x_0$, then we have lower bounds for  $\|A_I x_1\|=\|A_I' x_1\|$, i.e., $\delta_2$. 
 Observe that the first column of $A$ is independent of  the remaining columns of $A$.  
Note that   entries of $A'_{I}$ are i.i.d. Normal(0,1).
According to the random matrix theory of Wishart matrices,  with high probability, the singular values $\{\delta_i\}$ of the sub-matrix are bounded between $\sqrt{N_{I}}-\sqrt{n}$ and $\sqrt{N_{I}}+\sqrt{n}$. 
 More precisely, 
 \[
 \mathbb{P}(\sqrt{N_I}-\sqrt{n}-t\le \delta_2\le \delta_n\le \sqrt{N_I}+\sqrt{n}+t)\ge 1-2e^{-t^2/2},\; t\ge 0,
 \]
 see Eq.~(2.3)~\cite{rudelson2010non}.
  Together with 
$ N_I^{-1/2}\|b_I\|\le  c_6(\pi/6)^{1/2}N_I/N$ in Prop.~\ref{BnormA} and  Eq.~(\ref{closeness1}),
we have
 \[
\|x_0x_0^\top -x_{min}x_{min}^\top\|/\sqrt{2}\le  \|b_I\|/(\sqrt{N_I}-\sqrt{n}-t)\]\[\le   c_6\sqrt{\pi/6}(N_I/N) (1-\sqrt{(n-1)/N_I}-\sqrt{t/N_I})^{-1}.
 \]
Let  $ c_7:= c_6  (1-\sqrt{(n-1)/N_I}-\sqrt{t/N_I})^{-1} $, then we have the following result.
\begin{prop} Suppose that $\sqrt{N_I}>\sqrt{n}+t$ for $t>0$. 
Then with high probability
\[
\frac{1}{\sqrt{2}}\|x_{min}x_{min}^\top-x_0 x_0^\top \|\le (N_I/N) c_7\sqrt{\pi/6}. 
\]
for some constant $ c_7>0$, independent of $N$.
\end{prop}
 \begin{figure}[htbp]
\begin{center}
\includegraphics[width=0.32\textwidth]{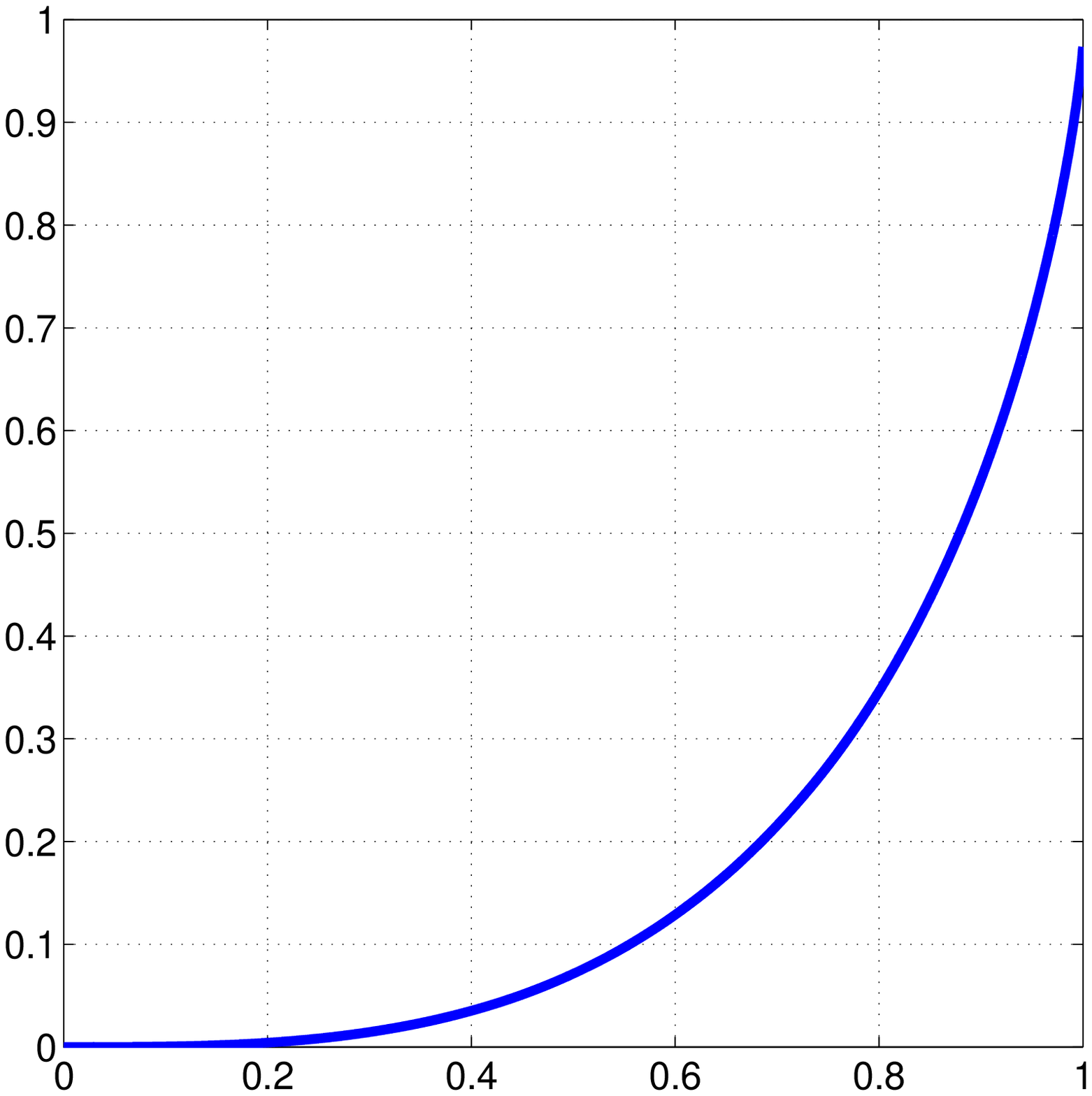}
\includegraphics[width=0.32\textwidth]{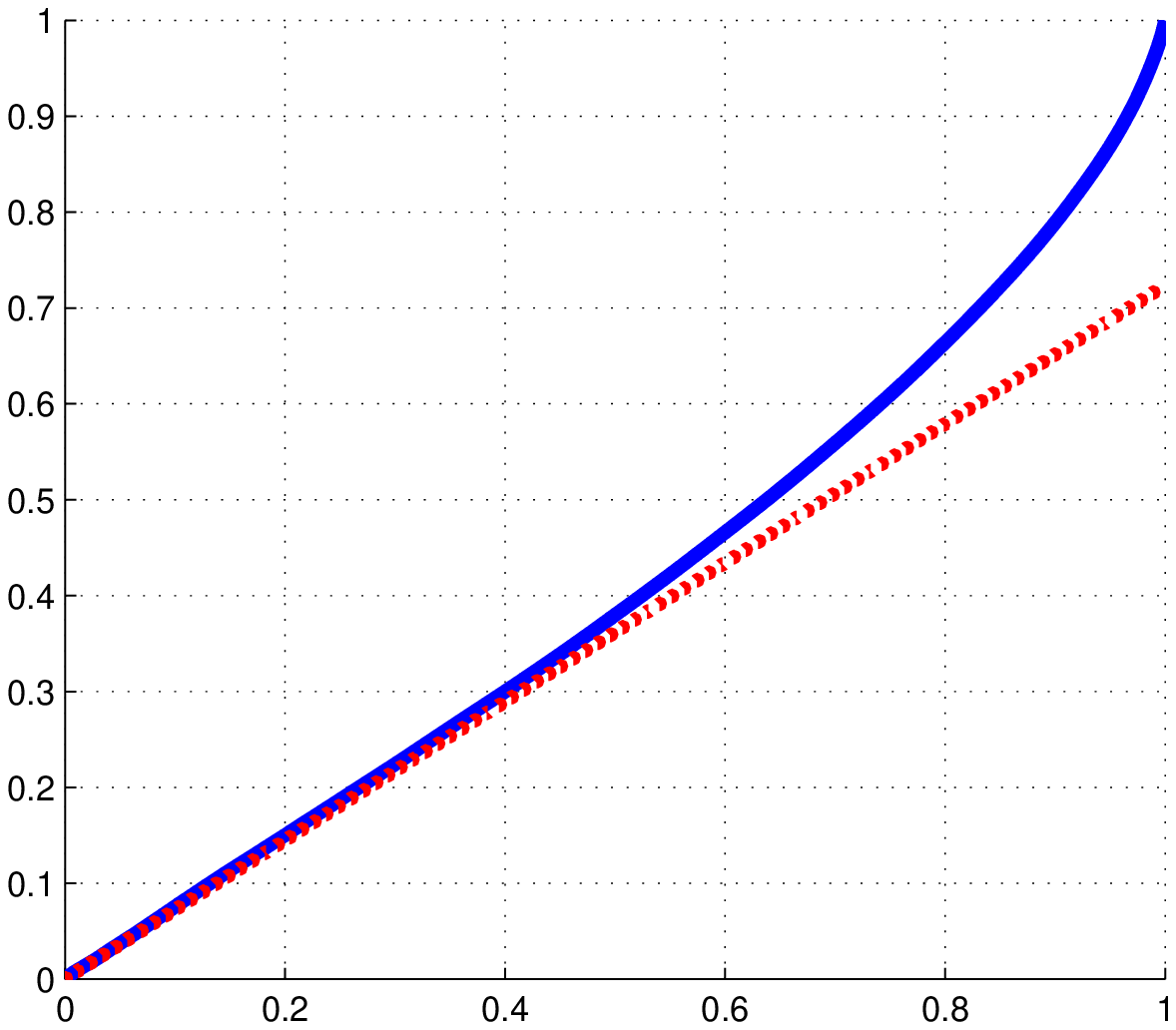}
\caption{{  Figures show $N^{-1}\|b_I\|^2$ and $N_I^{-1/2}\|b_I\|$ vs. $N_I/N$. The red dotted line is $\sqrt{\pi/6}(N_I/N)$ vs. $N_I/N$.}}
\label{bnorm}
\end{center}
\end{figure}
%

\begin{rem}\label{Simulation}
The following simulation illustrates
that  $x_{min}$ is a good initialization  $x^0$. 
Use the  alternating minimization of $x$ and $s$ to solve  the problem
\[
\min_s \min_{x} \| Ax-sb\|^2, \; |s|=1.
\]
Choose $A$ to be a Gaussian random matrix from $\mathbf{R}^{240\times 60}$ and $N_I=90$. Rescale both $x_0,x^*$ to unit vectors, where $x^*$ is the vector $x$ at the final iteration.  In Fig.~\ref{xmin}, the
 reconstruction error  is measured in terms of 
 \[ \|x_0x_0^\top  -x^*{x^*}^\top\|.\] 
 The  figures in Fig.~\ref{xmin} show the results of $100$ trials using two different  initializations. Obviously the singular vector is a good initialization.
 
 \begin{figure}[htbp]
\begin{center}
\includegraphics[width=0.32\textwidth]{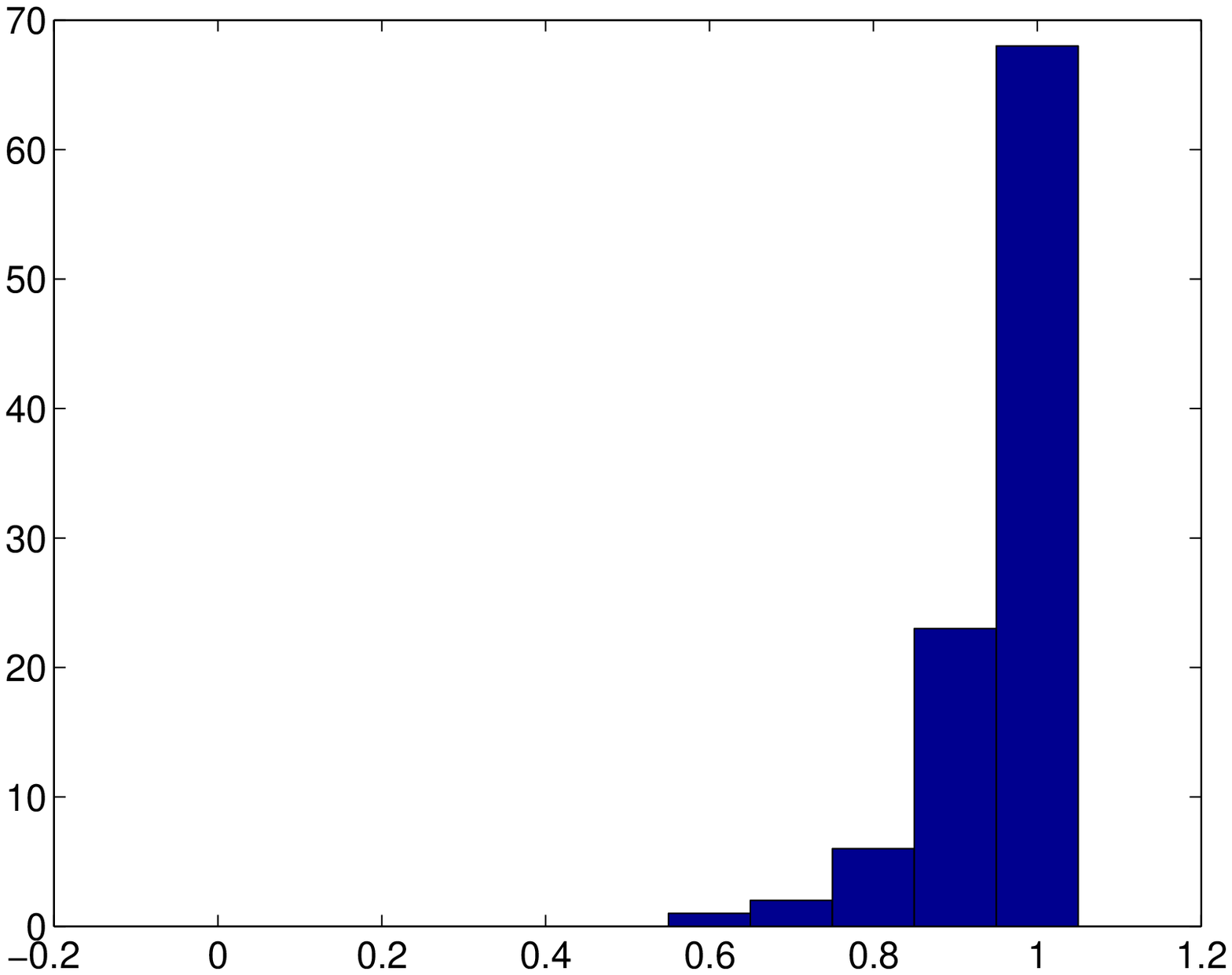}
\includegraphics[width=0.32\textwidth]{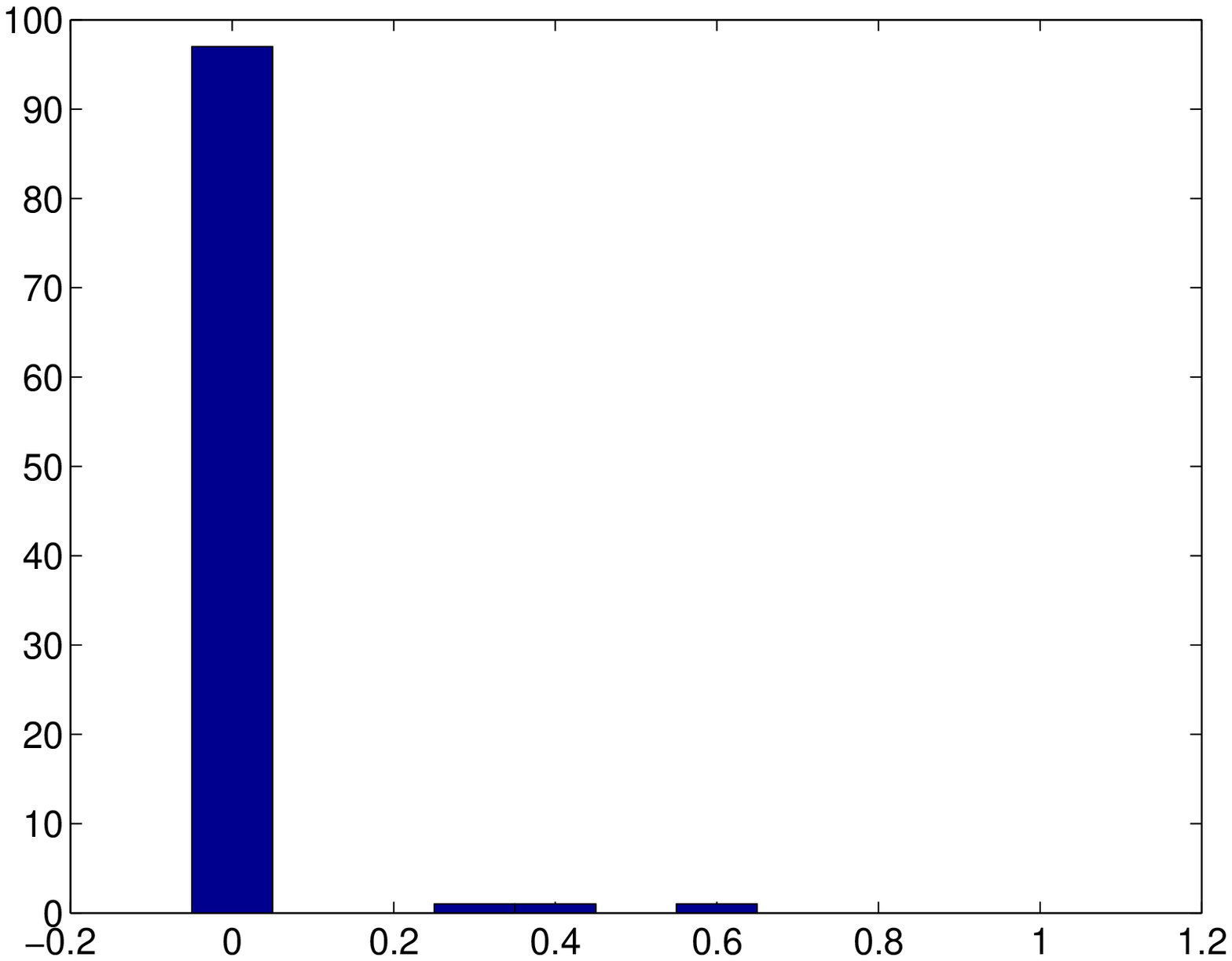}
\caption{Figures show   the histogram of the reconstruction error  via  the random initialization (left) and via the $x_{min}$ initialization (right).
}
\label{xmin}
\end{center}
\end{figure}

 \end{rem}

\subsection{ADM with rank-$r$ }\label{rankR}

Under some circumstances, the singular vector corresponding to the least singular value becomes a poor  initialization for $x_0$, for instance, in the presence  of noise. Empirically, we find that  the ADM with rank $r$ can alleviate the situation;
  see~\ref{NoiseEx}.

We propose the  rank-$r$ method:  
\begin{equation}
\min_{x\in \mathbf{R}^{n,r}} \frac{1}{2}\||Ax|-b\|^2,
\end{equation}
where  $|z|$ refers to the vector whose $i$-th entry is the vector norm of the $i$-th row of the matrix $z$, i.e., $(\sum_{j=1}^r z_{i,j}^2)^{1/2}$.
Note that 
when $r=n$, the  set $\{xx^\top : |Ax|=b\}$ is convex. 
In practical applications, we consider $r<n$ to save the computational load. 
Hence,  instead of  vectors $x,z$ in Eq.~(\ref{Wen}), we consider matrices
 $x\in \mathbf{R}^{n,r}$ and $z\in \mathbf{R}^{p,r}$  with $Ax=z$ in 
 the non-convex minimization problem,
 \begin{equation}\label{rankR}
\min_{z,x} L(z,x,\lambda),\end{equation}
 \begin{equation}
  L(z,x,\lambda):= \{ \frac{1}{2}\||z|-b\|^2+ tr(\lambda^\top (Ax-z)) +\frac{\beta}{2}\|Ax-z\|_F^2\}.
\end{equation}
Similar to Alg.~\ref{WenAlg}, we can adopt  the ADM consisting of $z,x,\lambda$-iterations to solve the non-convex minimization problem.
With  $\lambda$ fixed, the optimal matrices $z,x$ have the following explicit expression.
\begin{prop}
Suppose $(z,x)$ is a minimizer in Eq.~(\ref{rankR});  then, $(zV,xV)$ is also a minimizer for any orthogonal matrix  $V\in \mathbf{R}^{n,n}$. Moreover,
\begin{itemize} \item for each $z$ fixed, $x=A^\dagger (z-\beta^{-1} \lambda)$ is the optimal matrix. 
\item 
   Fixing $x$, write $u=Ax+\beta^{-1}\lambda$,  then the optimality of $z$ is  \[
z=\frac{u}{|u|}\frac{b+\beta |u|}{1+\beta}.
\]
\end{itemize}
\end{prop}
\begin{proof} We only prove the $z$-part. The $x$-part is obvious.  Because $L$ is separable in each row $z_i$ of $z$, then the optimization of $z_i$ can be solved via 
\[
\min_{z_i} \frac{1}{2} (\|z_i\|^2-2b_i  \|z_i\|+b_i^2)+\frac{\beta}{2} (\|u_i\|^2-2u_i\cdot z_i+\|z_i\|^2).
\]
The optimality of $z_i$ occurs if and only if $z_i$ parallels $u_i/\|u_i\|$. Let $z_i =\alpha_i u_i/\|u_i\| $ with $\alpha_i$ to be determined. Thus,
\[
\min_{\alpha_i} \frac{1}{2} (\alpha_i^2-2b_i  |\alpha_i |+b_i^2)+\frac{\beta}{2} (\|u_i\|^2-2\alpha_i \|u_i\|+\alpha_i^2).
\]
Then, $\alpha_i\ge 0$ and \[
\alpha_i-b_i+\beta (\alpha_i-\|u_i\|)=0,
\]
i.e., $\alpha_i=(1+\beta)^{-1} (b_i+\beta \|u_i\|)$ completes the proof.
\end{proof}

Empirical experimentation shows that the above ADM can usually yield  an optimal solution $x\in \mathbf{R}^{n,r}$ 
with rank not equal to one,  which does satisfy  $|Ax|=b$.    To recover the  rank-one matrix $x_0x_0^\top $,  we take the following  steps. First,
we  standardize   $A$ to be a matrix consisting of orthogonal columns via QR or SVD factorizations, 
such that $I_{n\times n}$ lies in the range of $\mathcal{A}$. Indeed, 
\[
\sum_{i=1}^N a_ia_i^\top =A^\top A=I_{n\times n}.
\]
When $y\in \mathbf{R}^{n,r}$, we have $b^2\cdot e= tr(yy^\top)=\|y\|^2_F$.
Hence, the norm $\|y\|^2_F=\sum_{i=1}^n \sigma_i(y)^2$ remains constant for all the feasible solutions $\{y: |Qy|=b\}$, where $\sigma_i(y)$ refers to the singular values of $y$. 
Second, consider the objective function to retrieve the matrix $y$ with  the maximal  leading singular value,
\begin{equation}\label{sigma1}
\min_y \left(\frac{1}{2}\| Qy-z\|_F^2-\gamma \sigma_1(y)\right), 
\end{equation}
where $\gamma>0$ is some parameter to balance the fidelity $|Qy|=|z|=b$ and the maximization of the leading singular value $\sigma_1(y)$. \footnote{ In experiments, we choose $\gamma=0.01$.}
Since the leading singular value of $y$ is maximized,  there is no guarantee that we can always obtain the global optimal solution. 

\begin{prop}

Write $Q^\top z$ in the SVD factorization, \[ Q^\top z=U_z D_zV_z^\top.\] 
Then the optimal matrix $y$  in Eq.~(\ref{sigma1}) is $y=U_zD_yV_z^\top$ in the SVD factorization, where
 $D_y=D_z+\gamma e_1 e_1^\top$. 
\end{prop}
\begin{proof}
Observe that 
\[
\frac{1}{2}\| Qy-z\|_F^2-\gamma \sigma_1(y)=\frac{1}{2}\|U_z^\top Q U_yD_yV_y^\top V_z- D_z\|_F^2-\gamma D_y(1,1),
\]
where $D_y(1,1)$ refers to the (1,1) entry of the diagonal matrix $D_y$.
Due to  the rotational invariance of  the Frobenius norm, then the first term achieves its minimum when 
\begin{equation}\label{DD} U_z^\top Q U_yD_yV_y^\top V_z=D_z+\alpha e_1 e_1^\top \end{equation} and   $\alpha$ is the minimizer of
\[
\min_\alpha \left(\alpha^2/2-\gamma (D_z(1,1)+\alpha)\right), \; i.e., \alpha=\gamma.
\]
 Also,  Eq.~(\ref{DD}) yields $U_z^\top Q U_y=I$ and $V_y^\top V_z=I$,
which completes the proof.
\end{proof}
\begin{rem}
Suppose  that 
$|Ax|=b$ for some $x\in \mathbf{R}^{n,r}$. Then, the minimizer of \[
\min_{x} \left(\frac{1}{2}\| |Ax|-b\|^2-\gamma \sigma_1(x )\right)
\]
is $(1+\gamma) x_0$. Indeed, consider $x=\alpha x_0$. Then, 
\[
\alpha=arg\min_\alpha \left(\frac{1}{2} (\alpha-1)^2 \|b\|^2-\gamma \alpha\right)=1+\gamma.
\]
\end{rem}

In Eq.~(\ref{rankR}),  replacing $A$ with $Q$ and replacing the term 
$\frac{\beta}{2}\|Ax-z\|_F^2$ with  \[
\beta\left(\frac{1}{2}\| Qy-z\|_F^2-\gamma \sigma_1(y)\right), 
\]
then we adopt the ADM to 
 retrieve a rank-one solution.

\begin{alg}
\begin{enumerate} Initialize a random matrix $y\in \mathbf{R}^{n,r}$ and $\lambda^0=0_{N,r}\in \mathbf{R}^{N,r}$. 
Repeat the following steps, $k=1,2,\ldots $.  Then let  the  solution $x^*$ be the first column of $U_z$, i.e., the singular vector corresponding to the maximal singular value.  
\item  $z$-iteration:
\[
u=Qy^{k}+\lambda^k\beta^{-1},\; 
z^{k+1}=\frac{u}{|u|}\frac{b+\beta |u|}{1+\beta},
\]
\item $\lambda$-iteration:
\[
\lambda^{k+1}=\lambda^k+\beta (Qy^k-z^{k+1}),
\]
\item $y$-iteration:
\[
U_zD_z V_z=Q^\top (z^{k+1}-\lambda^{k+1}\beta^{-1}),\; 
y^{k+1}=U_z (D_z+\gamma e_1).
\]

\end{enumerate}

\end{alg}
\subsection{Standardized frames with equal norm}
In the simulations (section~\ref{Failure1}), we will show the importance of the unit norm condition $\|a_i\|=1$ for $i=1,\ldots, N$ in the ADM approach.  
When the QR factorization  is used to generate an equivalent standardized matrix consisting of rows $\{a_i\}_{i=1}^N$,
 the sensing vectors $\{a_i\}$ do not have equal norm  in general.
 
 The following theorem  states  that we can  standardize  $A$ to obtain  an orthogonal  matrix  $Q$ whose rows have equal norm.   
The proof is given in the appendix.

\begin{thm}\label{Standard}
Given a matrix $A\in \mathbf{R}^{N\times n}$ satisfying the rank* condition and $N> n$, we can find a unique   diagonal matrix $D$ with $D_{i,i}>0$,
such that  \[
D^{-1/2}A=QB,
\]
 and $Q$ is one \textit{standardized} matrix, which is 
 one projection matrix  with  $Q^\top Q=I_{n\times n}$,
 $(QQ^\top)_{i,i}=(n/N)$ for all $i$, where $B$ is   some $n\times n $ nonsingular  matrix.
  \end{thm}
  Here the diagonal value $n/N$ is  
the average of the norm $\|Q\|_F^2= tr(Q^\top Q)= tr(QQ^\top )=n$.  
Also, 
\[
N=\|D^{1/2} A\|_F^2=\|QB\|_F^2=\|B\|_F^2.
\] 
With the uniqueness of $D$, $Q$ is also determined uniquely up to the right multiplication of an orthogonal matrix. Indeed, $B$ is uniquely determined up to the left multiplication of an orthogonal matrix:  
\[
A^\top D^{-1/2}D^{-1/2} A=B^\top Q^\top QB=B^\top B.
\]

Recall that  $A$ satisfies the rank* condition if any   square $n$-by-$n$  sub-matrix  of $A$ is full rank.
When a matrix $A$ satisfies the rank* condition  then there exists
 no orthogonal matrix $V\in\mathbf{R}^{n\times n}$, such that  \begin{equation}\label{Block}AV=C=
\left(
\begin{array}{cc}
C_{1,1}  & 0   \\
 0 & C_{2,2}     
\end{array}
\right),
\end{equation}
where the $0$s refer to zero sub-matrices with size $(N-N_1)\times n_1$ and size $N_1\times (n-n_1)$ and $C_{1,1}$ is an $N_1\times n_1$ matrix. \footnote{ Otherwise, it is easy to see that one of the following  submatrices must be   rank deficient: (1) the top submatrix with entries $\{C_{i,j}: i,j=1,\ldots, n\}$ or (2)the bottom submatrix with entries $\{C_{i,j}: i=N-n+1,\ldots, N,\;  j=1,\ldots, n\}$.}
Furthermore, the condition ensures that 
 the norm of each row must be positive. It is easy to see that, with probability one,  Gaussian random matrices  satisfy the rank* condition.

\section{Experiments}
\subsection{ADM failure experiments}\label{Failure1}
Due to the nature of nonconvex minimization, the algorithm can fail to converge, which is indeed observed in the following two simulations. 

First,
 let us denote the input data by $(A,b)$ with $b_i\neq 0$ and the unknown signal  by $x_0$. 
Mathematically, solving  problem (i) \[ |Ax_0|=b  \] is equivalent to solving   problem (ii)\[
b_i^{-1}|a_i\cdot x_0|=1.
\]
However, solving  these two problems via  the ADM~\cite{Wen} can   yield   different results.  
 
 Let $A$ be a  real Gaussian  random matrix,   $A\in \mathbf{R}^{N\times n}$. Let $b=|Ax_0|$.
    Rescale the system  by $b^{-1}$, i.e.,   the input data becomes $(b^{-1}A, 1_{N\times 1})$, thus equal measurement values.
 Figure~\ref{Failure} shows  
the error  $\| |AA^\dagger z|-b\|$ at each iteration. Here we use the random initialization for $x^0$. 
\begin{figure}[htbp]
\begin{center}
\includegraphics[width=0.32\textwidth]{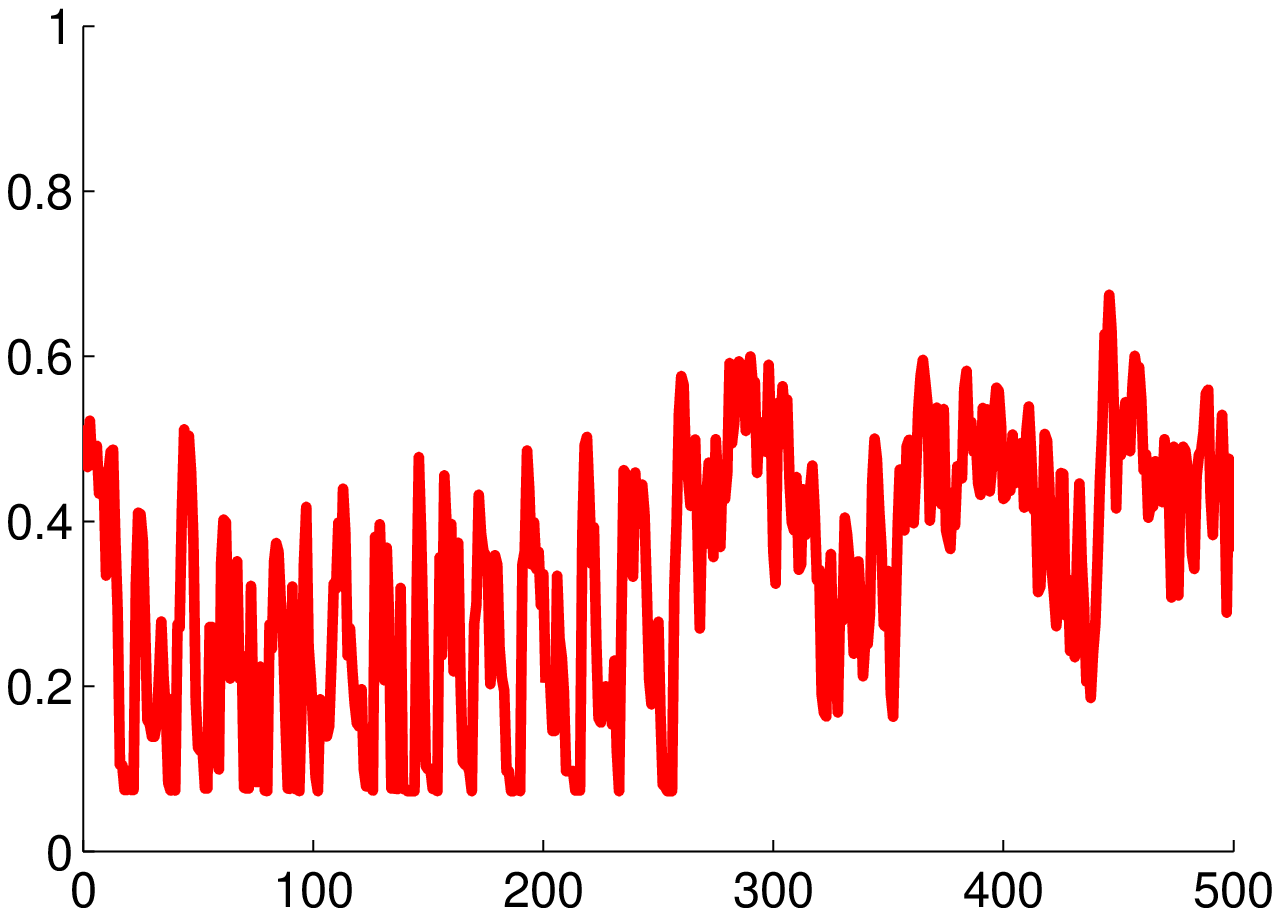}
\includegraphics[width=0.32\textwidth]{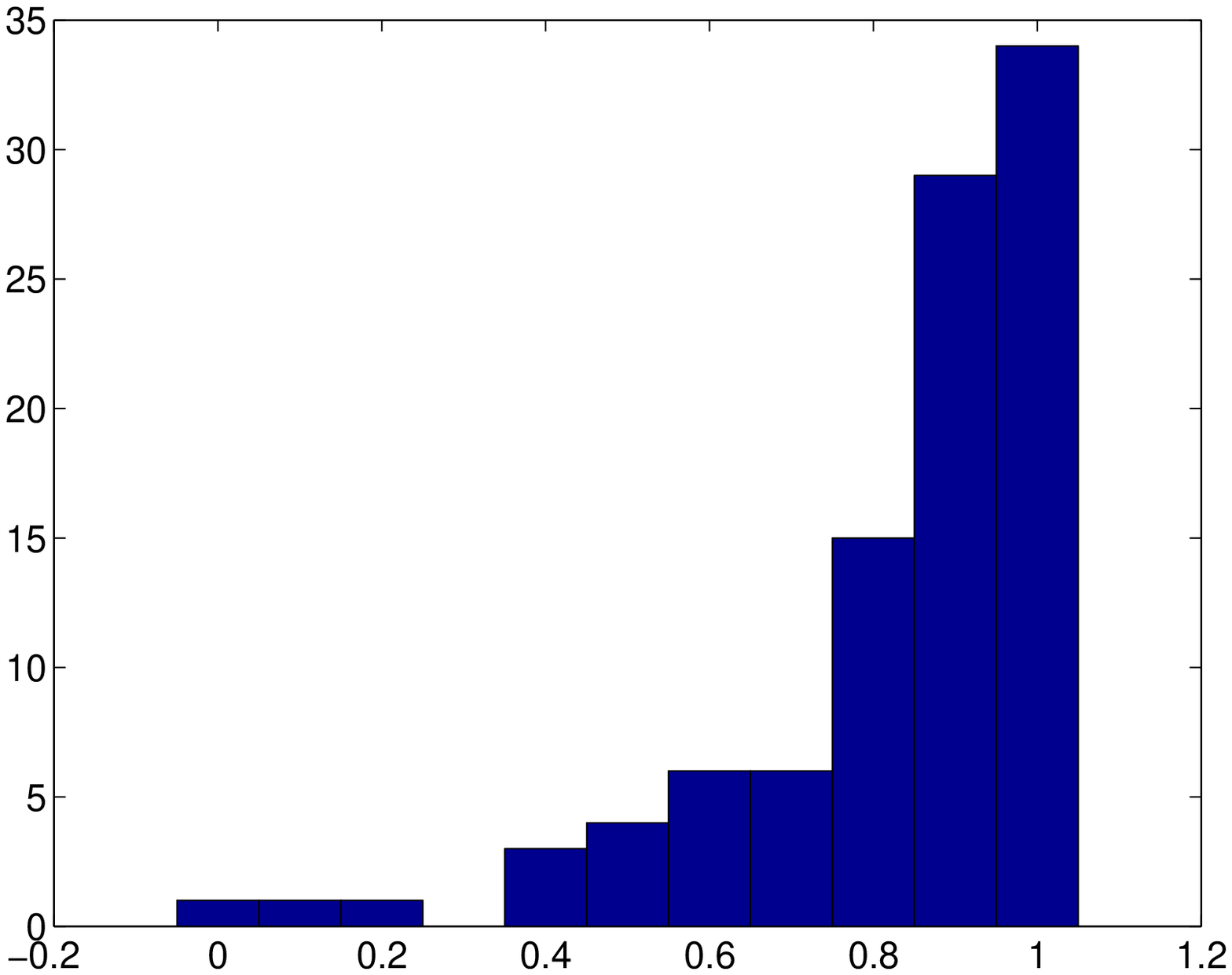}
\caption{{  The left figure shows the error $\| |AA^\dagger z|-b\|$ vs. the number of iteration via ADM with rank one~\cite{Wen}.  The right figure shows the histogram $\| x_0x_0^\top-x^*{x^*}^\top\|$ of  $100$ reruns.}}
\label{Failure}
\end{center}
\end{figure}

Second,  we demonstrate a few experiments where the ADM also fails to converge. The convergence failure  sheds
  light on the importance of    the two proposed  assumptions in Section~\ref{Recovery}.


We sort  a set of random generated sensing vectors   $\{a_i\in \mathbf{R}^{100}\}_{i=1}^{400}$, such that \[ \textrm{
 $|a_i\cdot x_0|\le |a_j\cdot x_0|$ for all $i<j$.}\] That is, the indices are sorted according to the values $b_i$. 
We consider three different manners of  selecting $200$ sensing vectors  $\{a_i\}$:  (1)
the vectors with the smallest  indices,(2) the vectors with the largest indices, and
 (3) a combination with $199$ small indices and one large index. Finally, we compare these results with  the result using  a random selection of sensing vectors,  as shown in Fig.~\ref{Selection}. Here, we fix rank $r=1$ and $\beta=0.01$. Clearly, the combination with smaller indices and larger indices performs best. 

\begin{figure}[htbp]
\begin{center}
\includegraphics[width=0.32\textwidth]{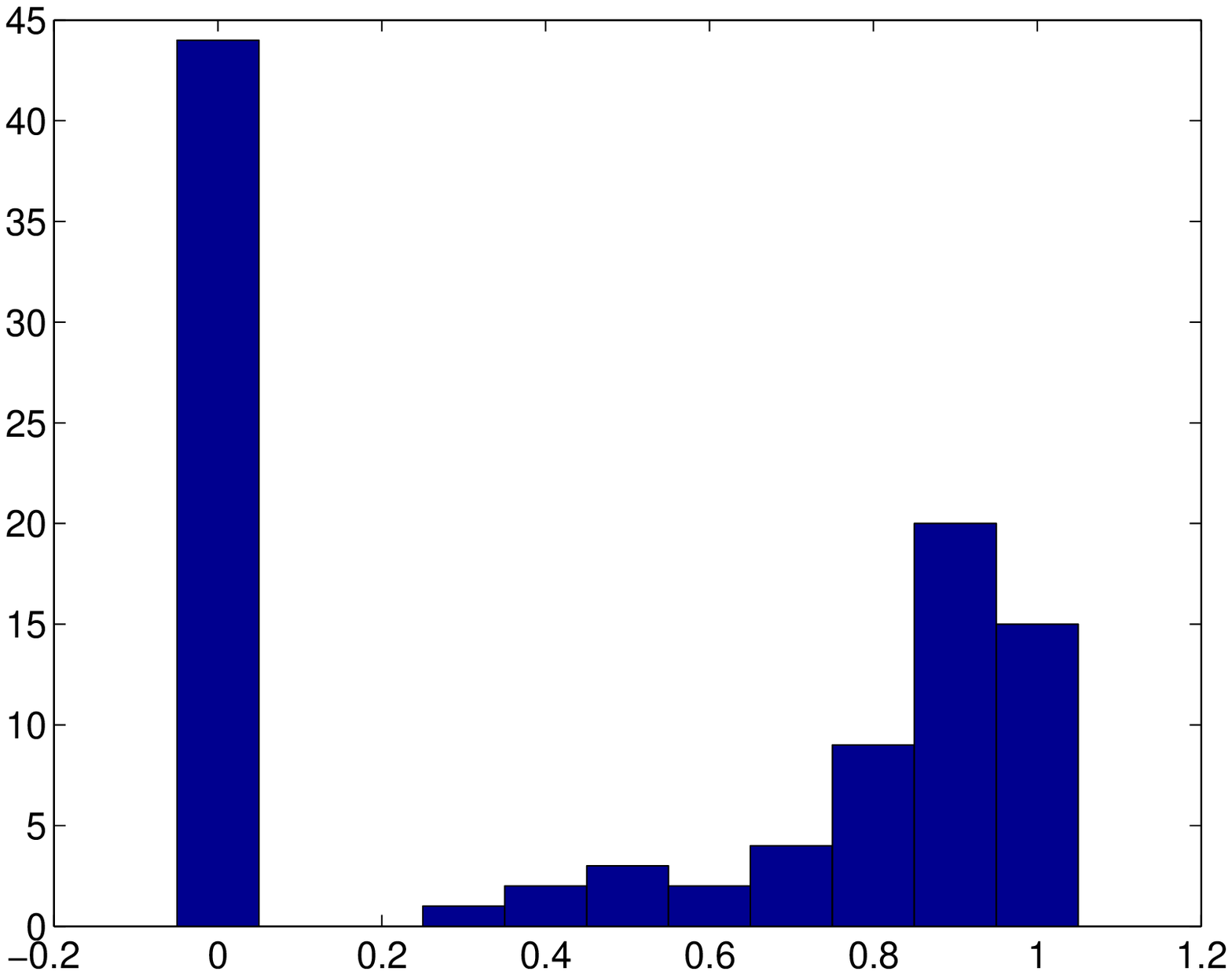}
\includegraphics[width=0.32\textwidth]{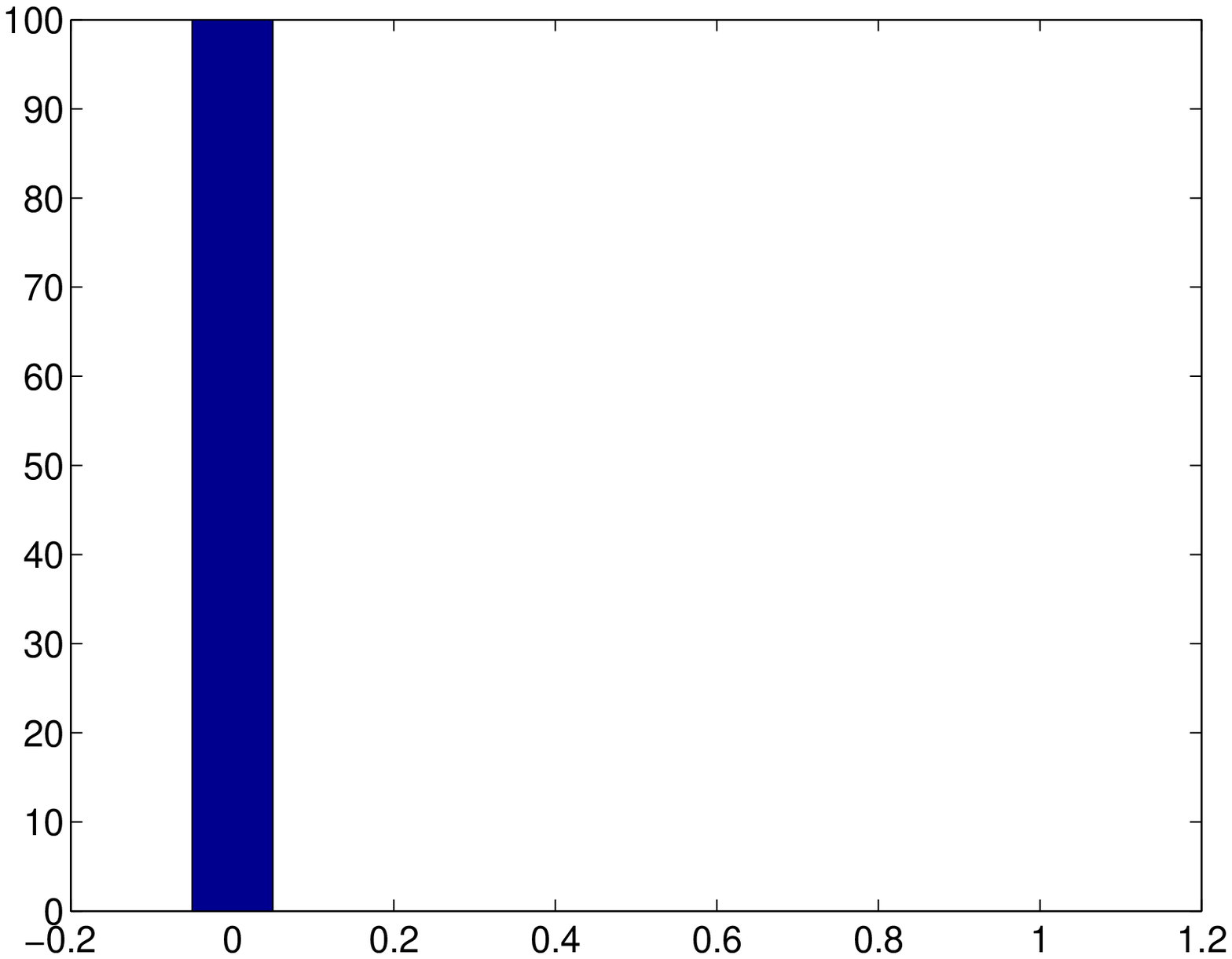}
\includegraphics[width=0.32\textwidth]{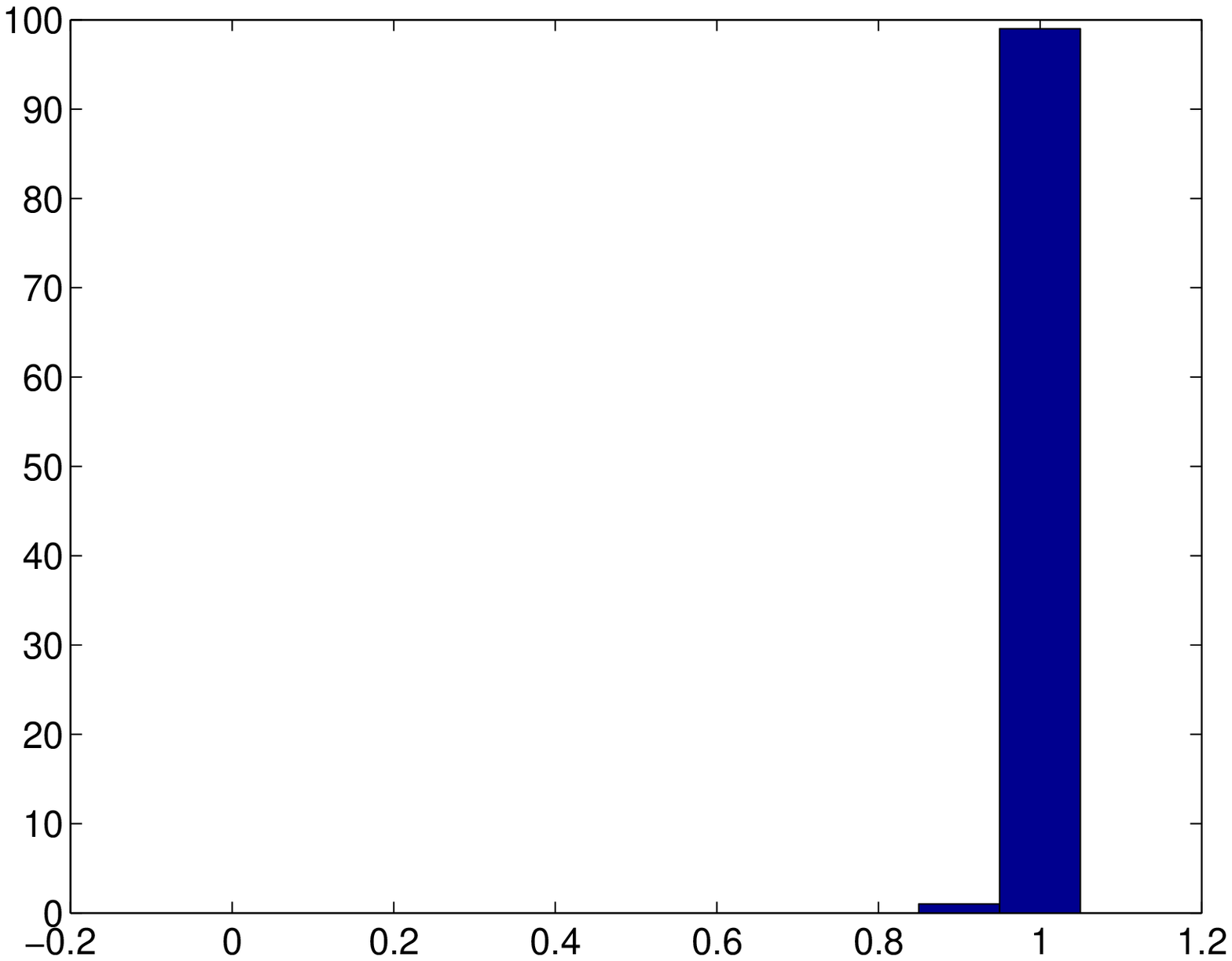}
\includegraphics[width=0.32\textwidth]{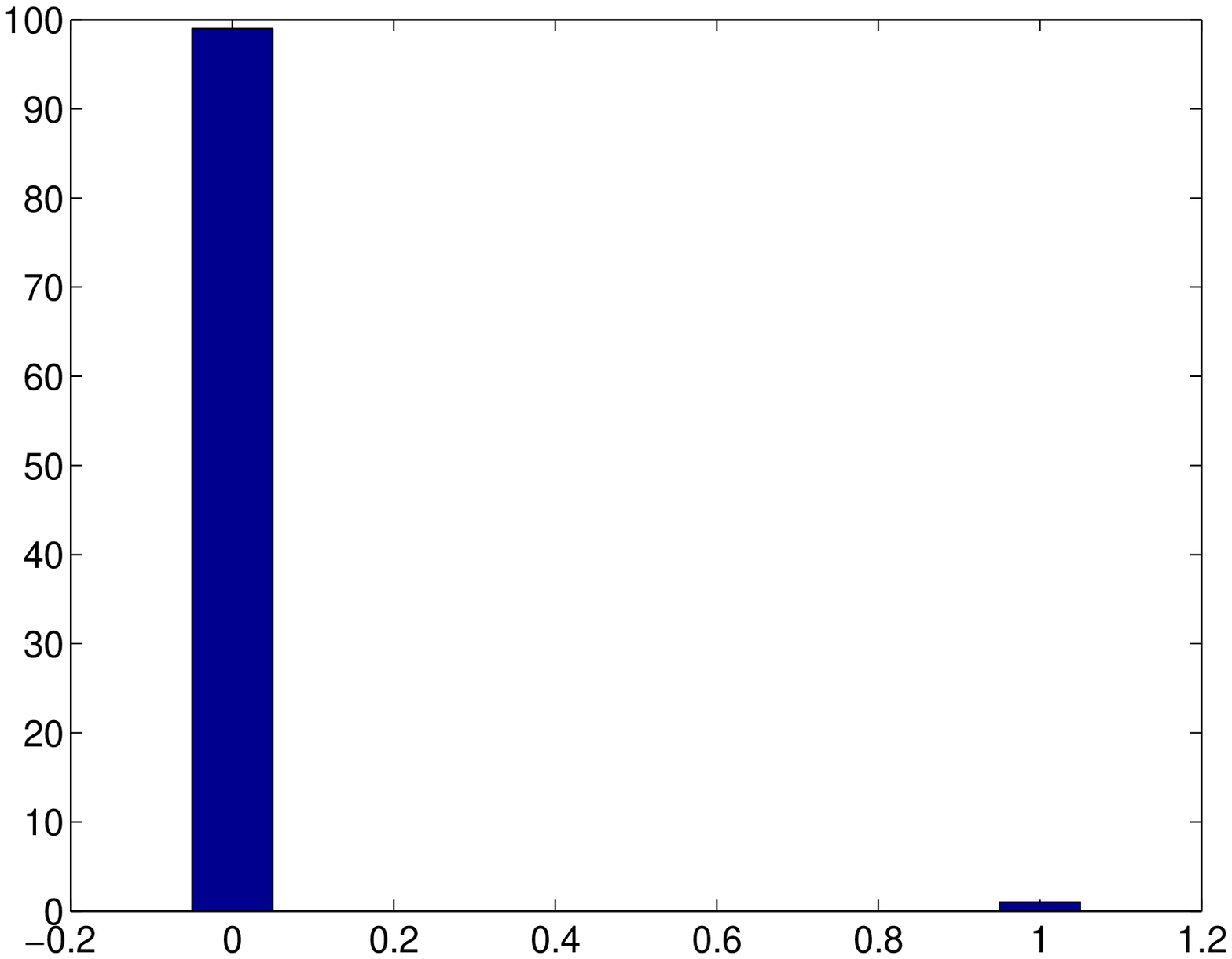}

\caption{{    Figure shows the histogram  $\| x_0x_0^\top-x^*{x^*}^\top\|$ under four different sets of $\{a_i\}_{i=1}^{400}$.   In the top row, Left, middle and right subfigures show the results with $\{a_i\}_{i=1}^{200}$, with $\{a_i\}_{i=1}^{199}\cup \{a_{400}\}$ and with $\{a_i\}_{i=201}^{400}$. The bottom subfigure shows the result when randomly sampling $200$ sensing vectors. }}
\label{Selection}
\end{center}
\end{figure}

%
%
%
%
%
%
%
%

\subsection{Comparison experiments with noises}\label{NoiseEx}
In this subsection, we demonstrate the performance of the ADM with $r=1$ and $r>1$ on a number of simulations, where  
Gaussian white noise is added.  
The noise-corrupted
 data, $b$, is generated,
\[
b^2=\max ((Ax_0)^2+noise,0).
\]
The signal-to-noise ratio is defined by
\[
SNR=10\log_{10}\frac{\|Ax_0\|^2}{\|noise\|_F}.
\]

In Fig.~\ref{Noise}, we consider  $A$ to be a real Gaussian  random matrix with $N=2n$.
We rerun  the experiments $200$ times to test the effect of random initialization.
The first row shows the histogram result with $n=30$ and  $noise=0$. All the algorithms with $r=1,2,$ and $3$ work well.
The second row shows the histogram result with $n=30$ and   $SNR=29$.
Here, we use $\beta=0.001$. Obviously  the algorithms with $r>1$ have better performances. 
\begin{figure}[htbp]
\begin{center}
\includegraphics[width=0.32\textwidth]{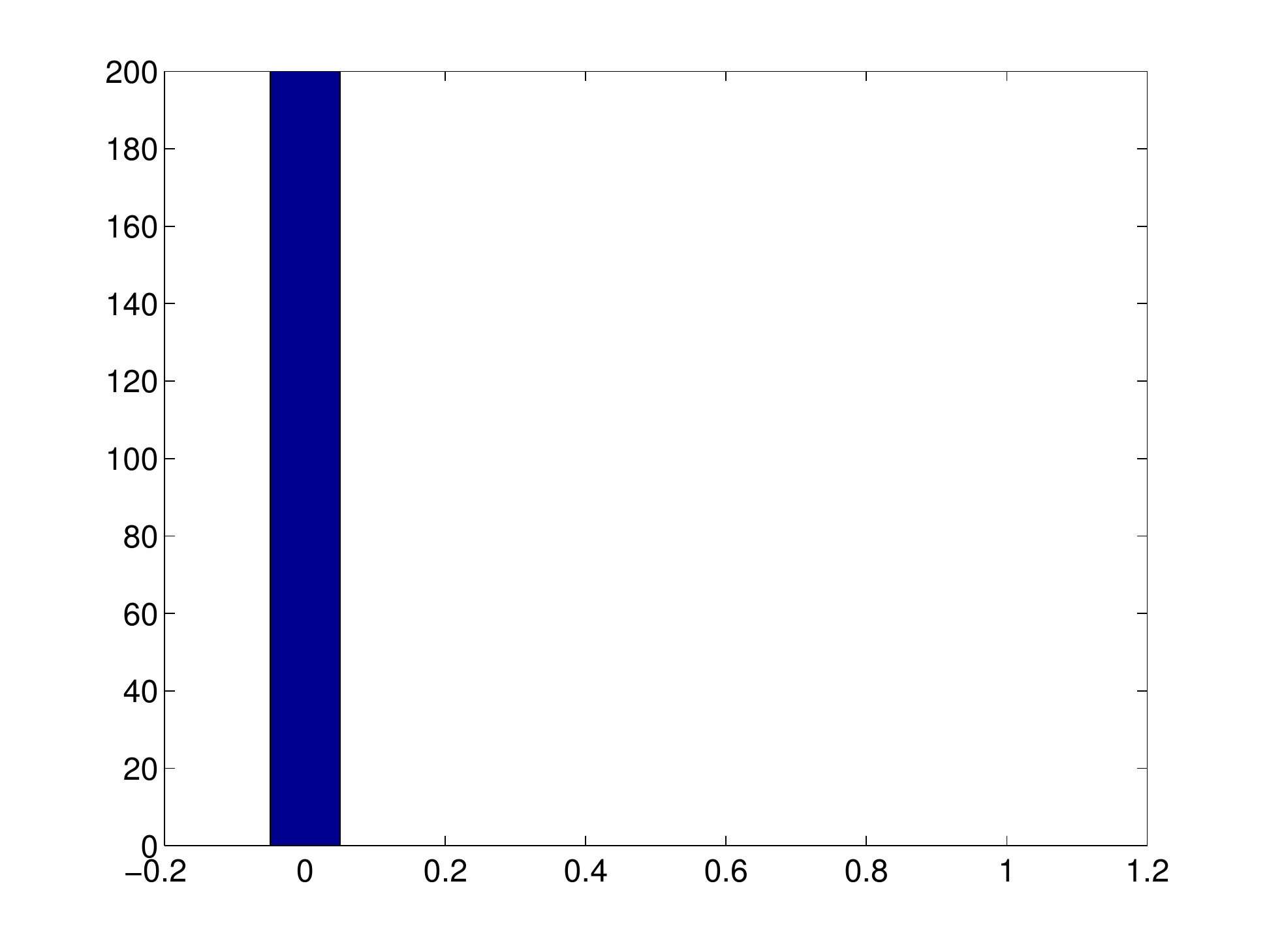}
\includegraphics[width=0.32\textwidth]{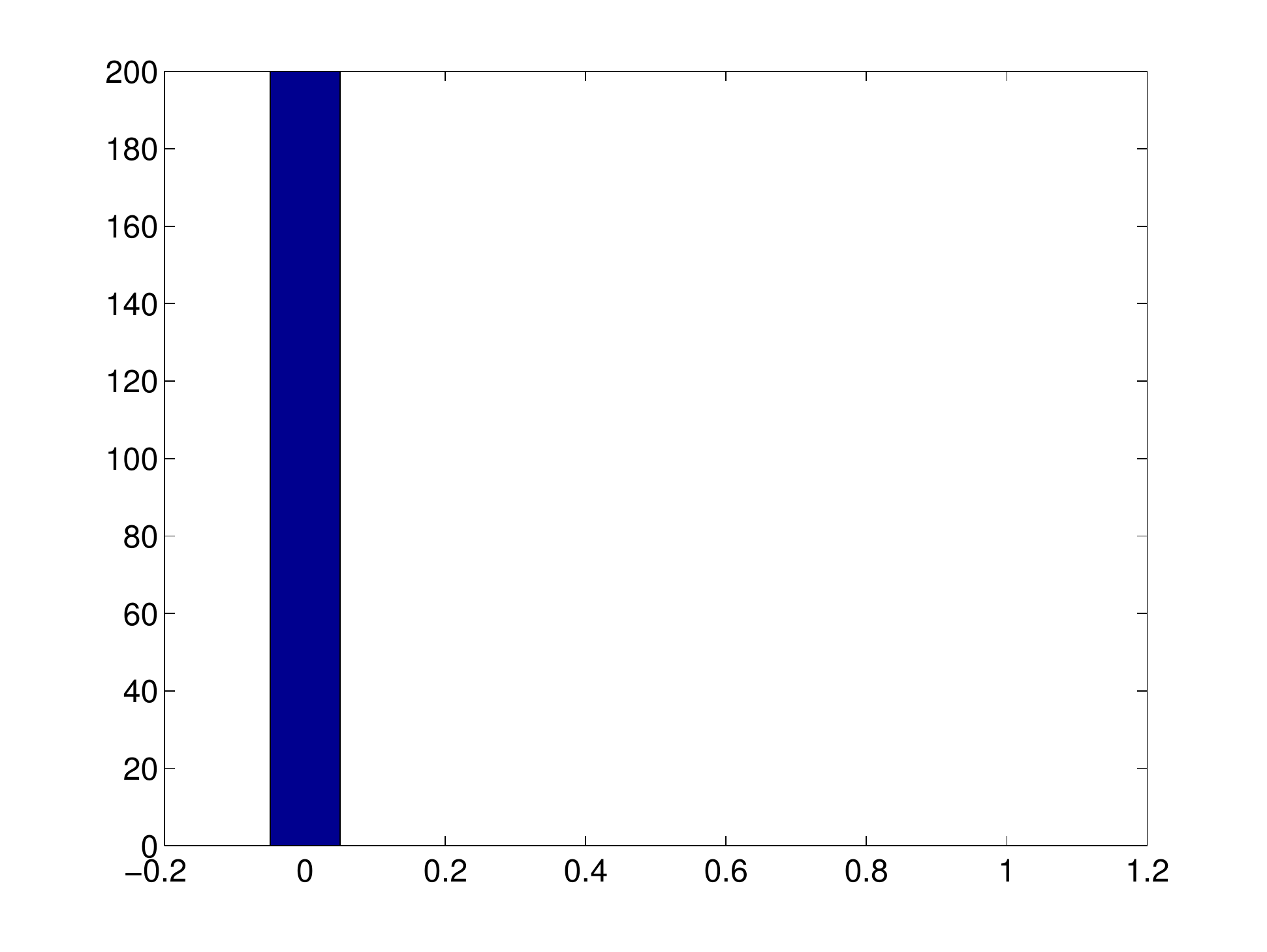}
\includegraphics[width=0.32\textwidth]{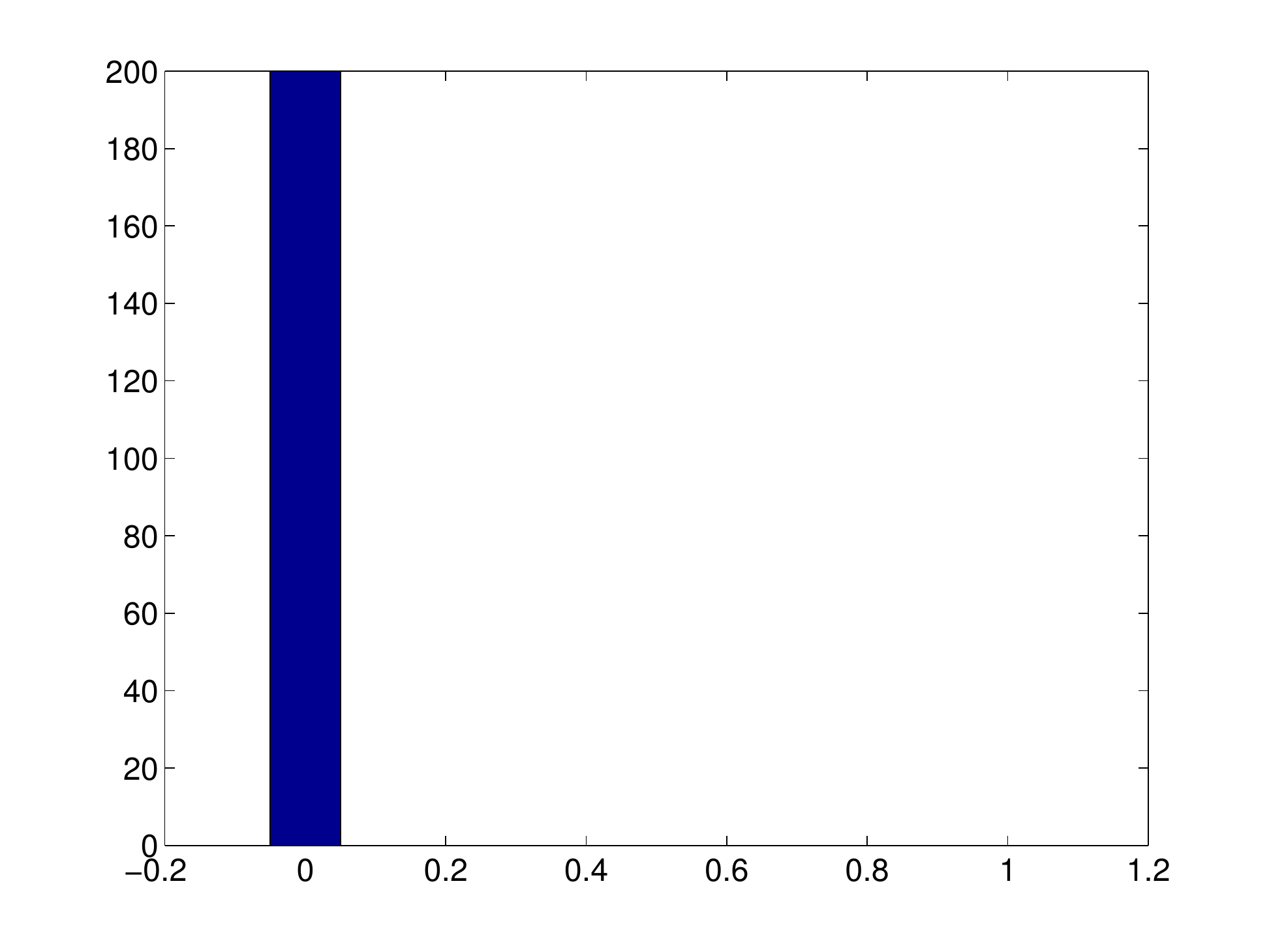}
\includegraphics[width=0.32\textwidth]{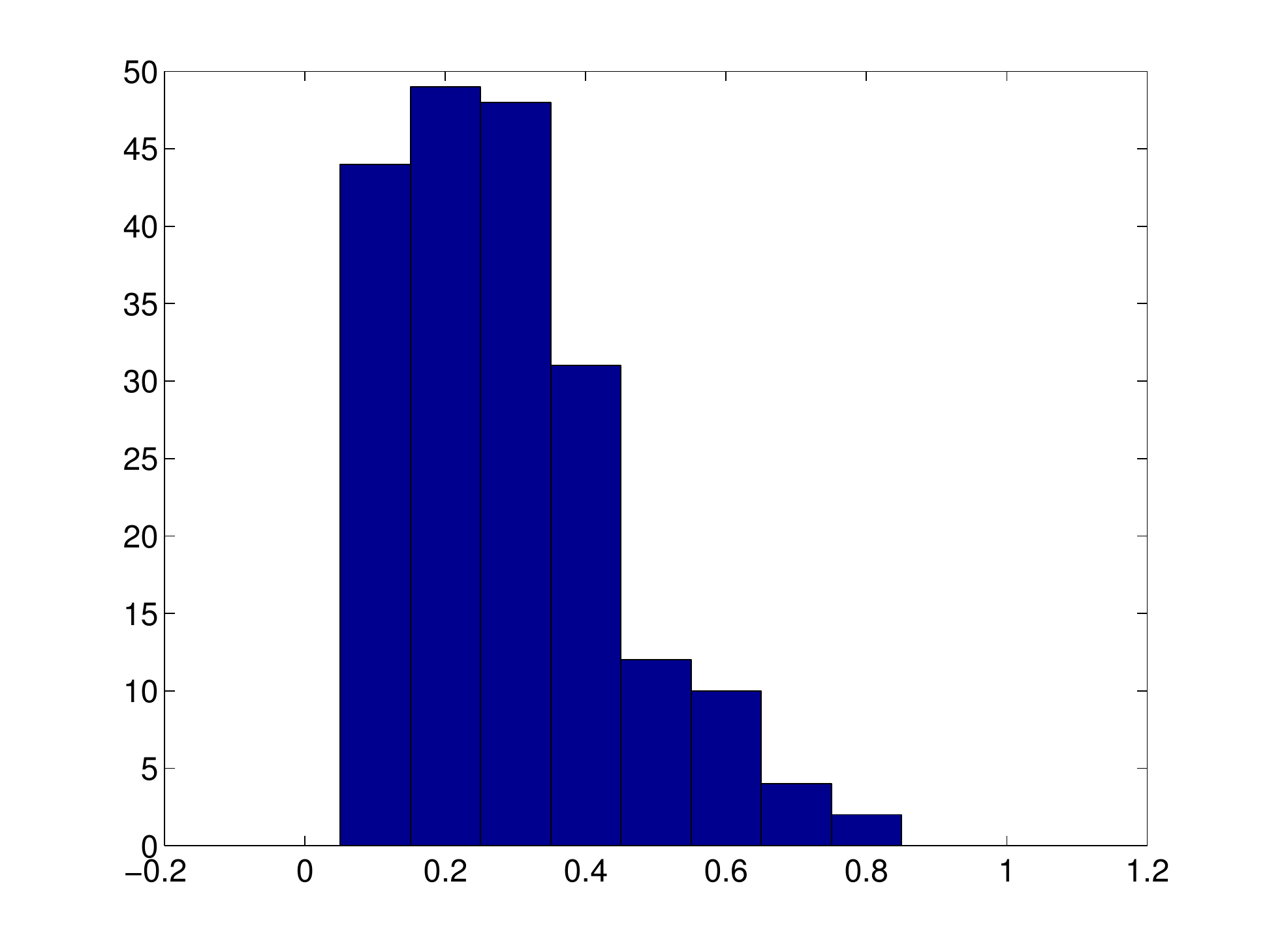}
\includegraphics[width=0.32\textwidth]{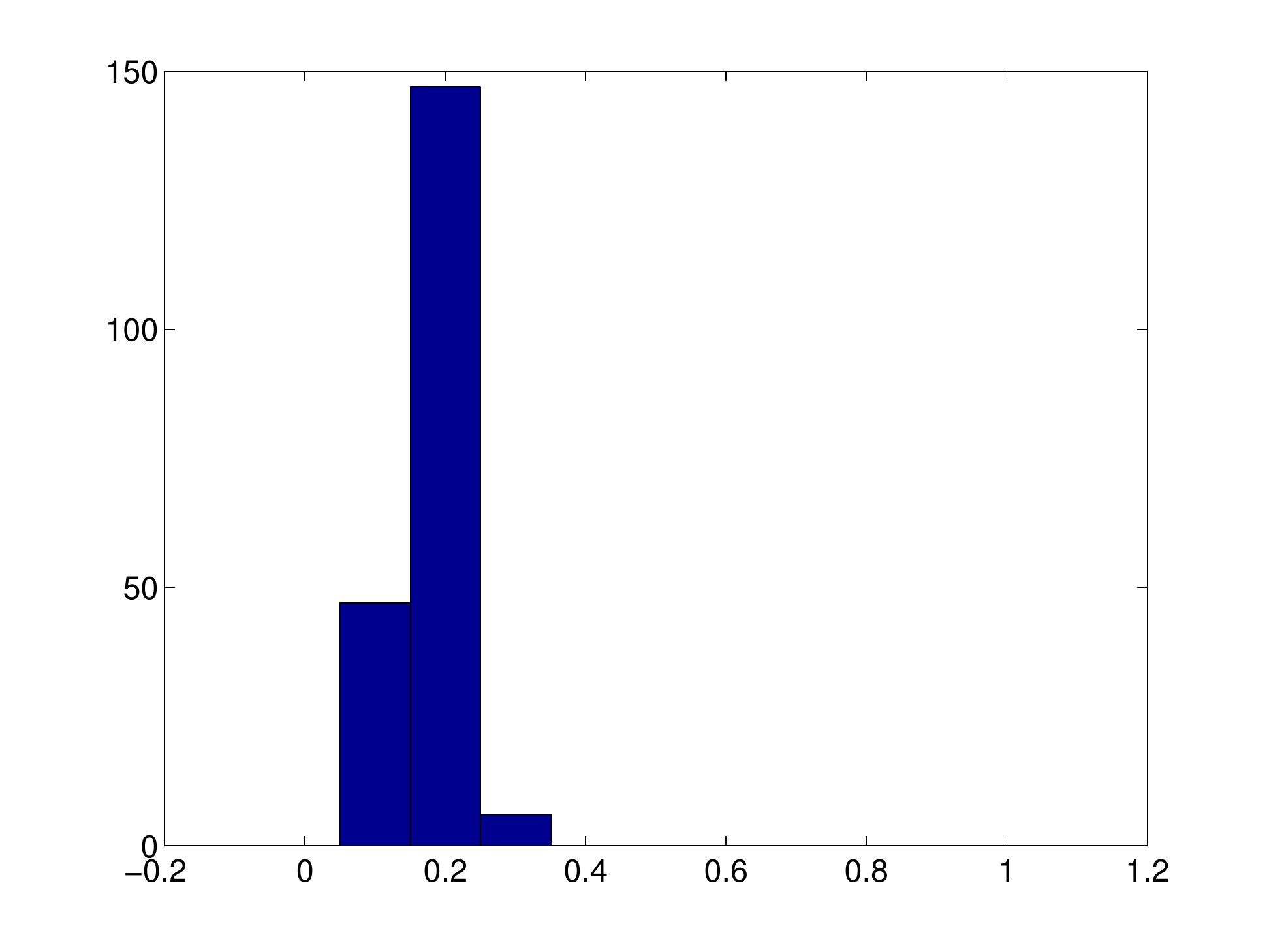}
\includegraphics[width=0.32\textwidth]{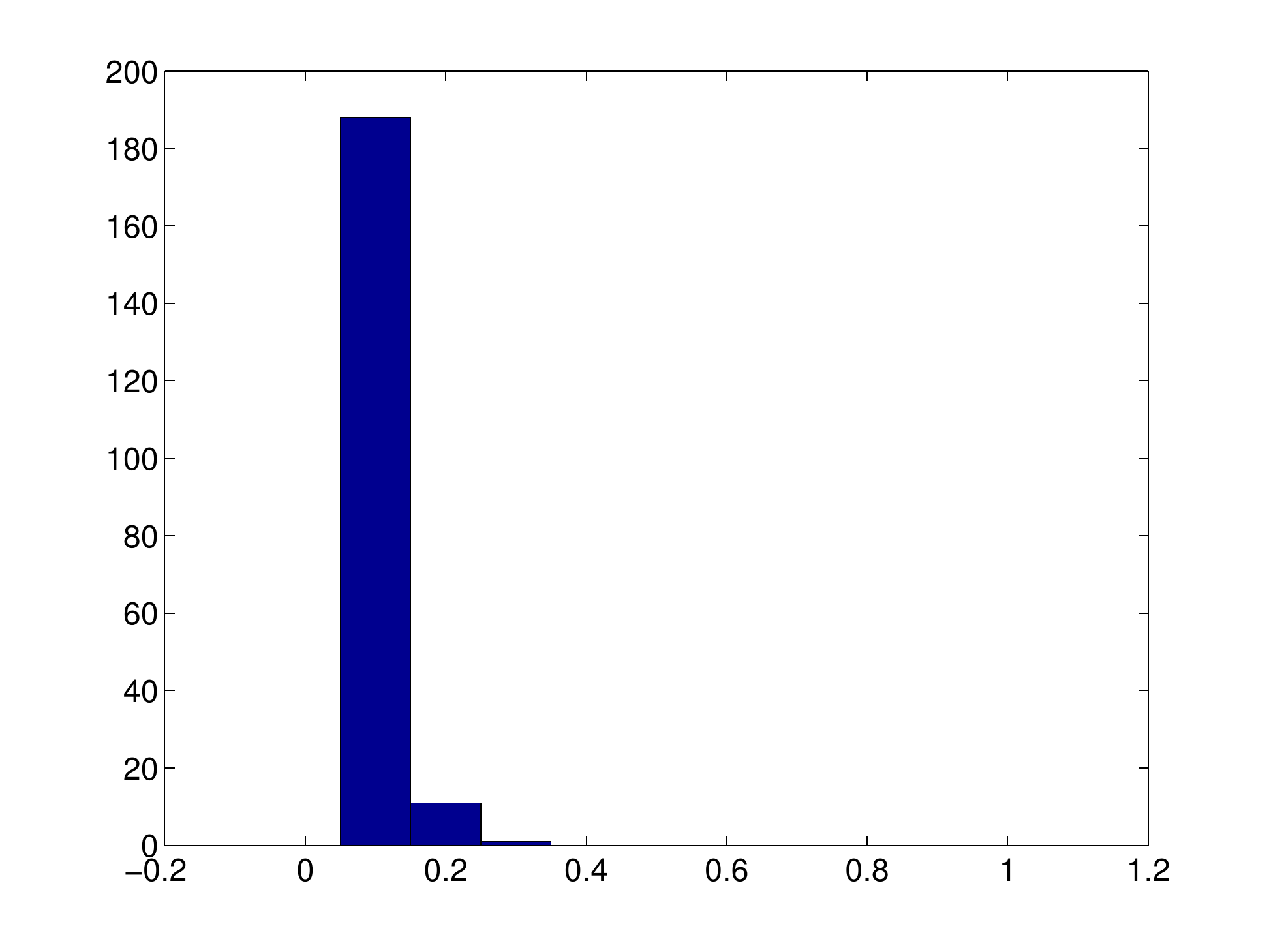}

\caption{{  Figure shows the histogram  $\| x_0x_0^\top-x^*{x^*}^\top\|$ under the noise effect.  Left, middle and right columns show the results with  rank  $r=1$, $r=2$ and  $r=3$.  Here we use the random initialization. }}
\label{Noise}
\end{center}
\end{figure}

Let  $n=30$, $N=3n$ with $\beta=0.01$.
In Fig.~\ref{Noise1}, we demonstrate the comparison between  the random initialization and the singular vector initialization, i.e.,  the initialization is chosen to be the singular vector corresponding to the least singular value of $A_{I} \in \mathbf{R}^{45\times 30}$.
  Data $b$ is generated with $noise=2\times 10^{-4}  \times Normal(0,1)$, $SNR=25dB$.
  Furthermore, with the presence  of noise, when ADM  with $r=1$ is employed, the difference between the two initializations  is very little, in contrast to  the simulation result shown in  Remark~\ref{Simulation}.

\begin{figure}[htbp]
\begin{center}
\includegraphics[width=0.32\textwidth]{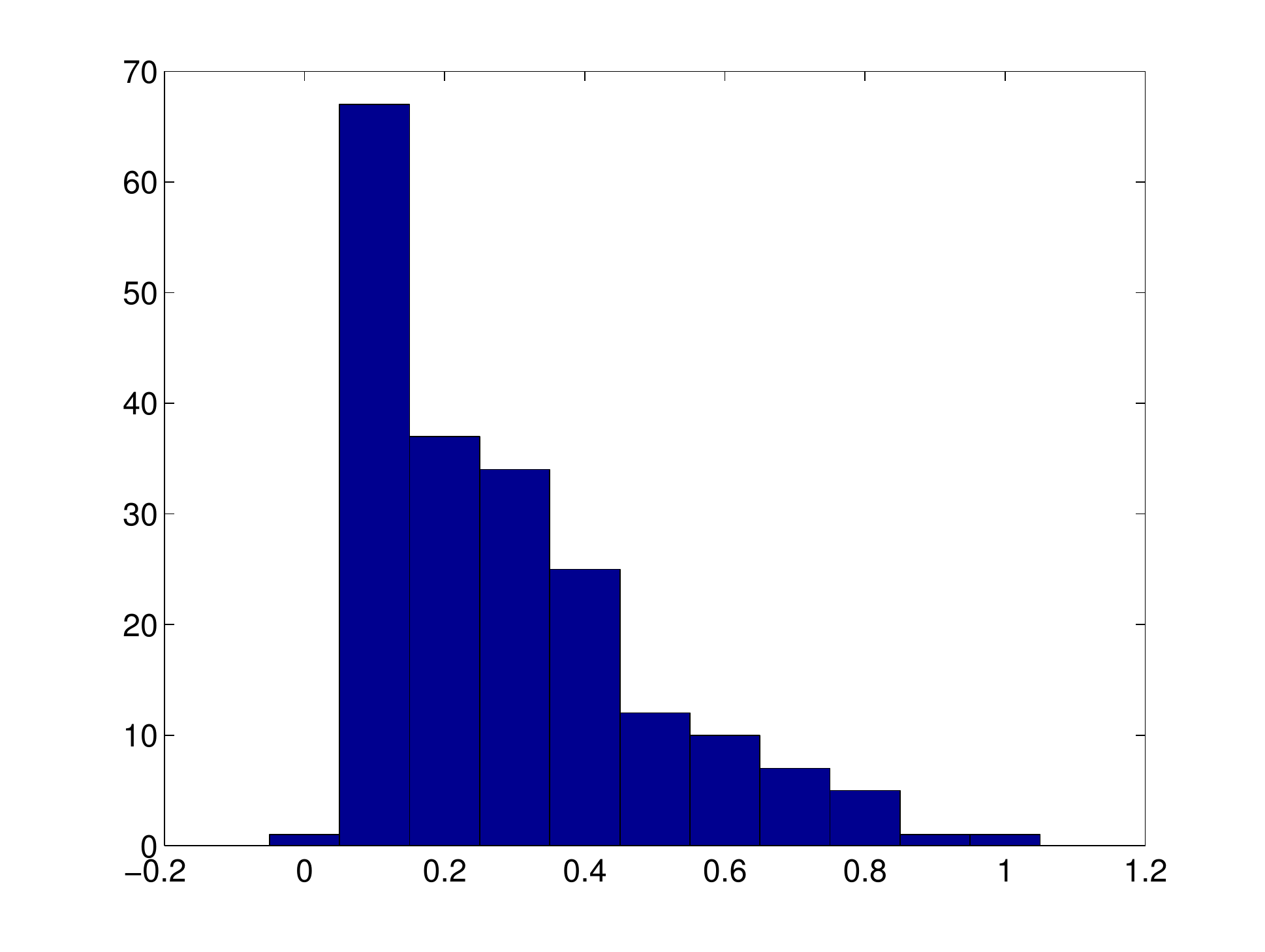}
\includegraphics[width=0.32\textwidth]{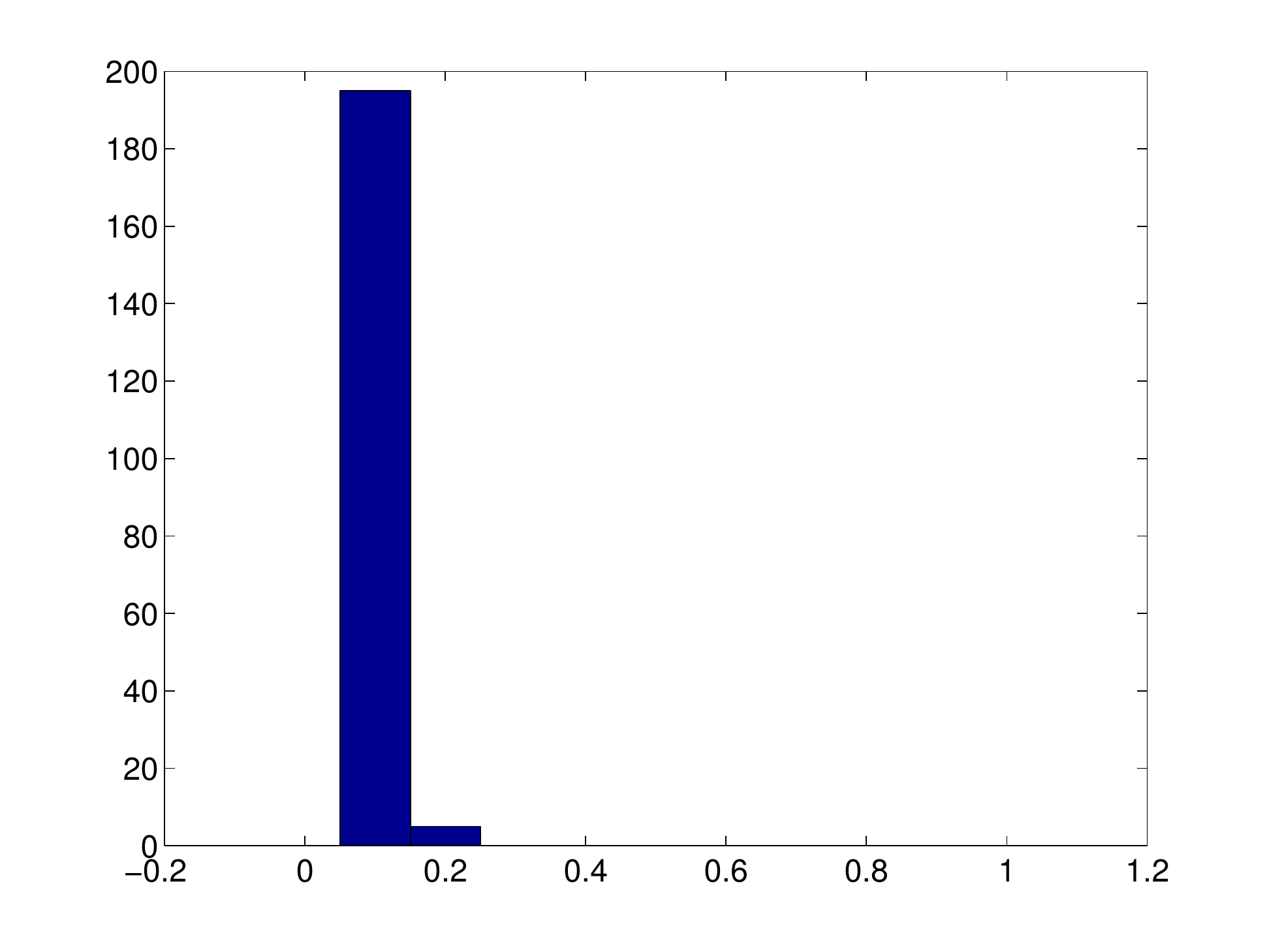}
\includegraphics[width=0.32\textwidth]{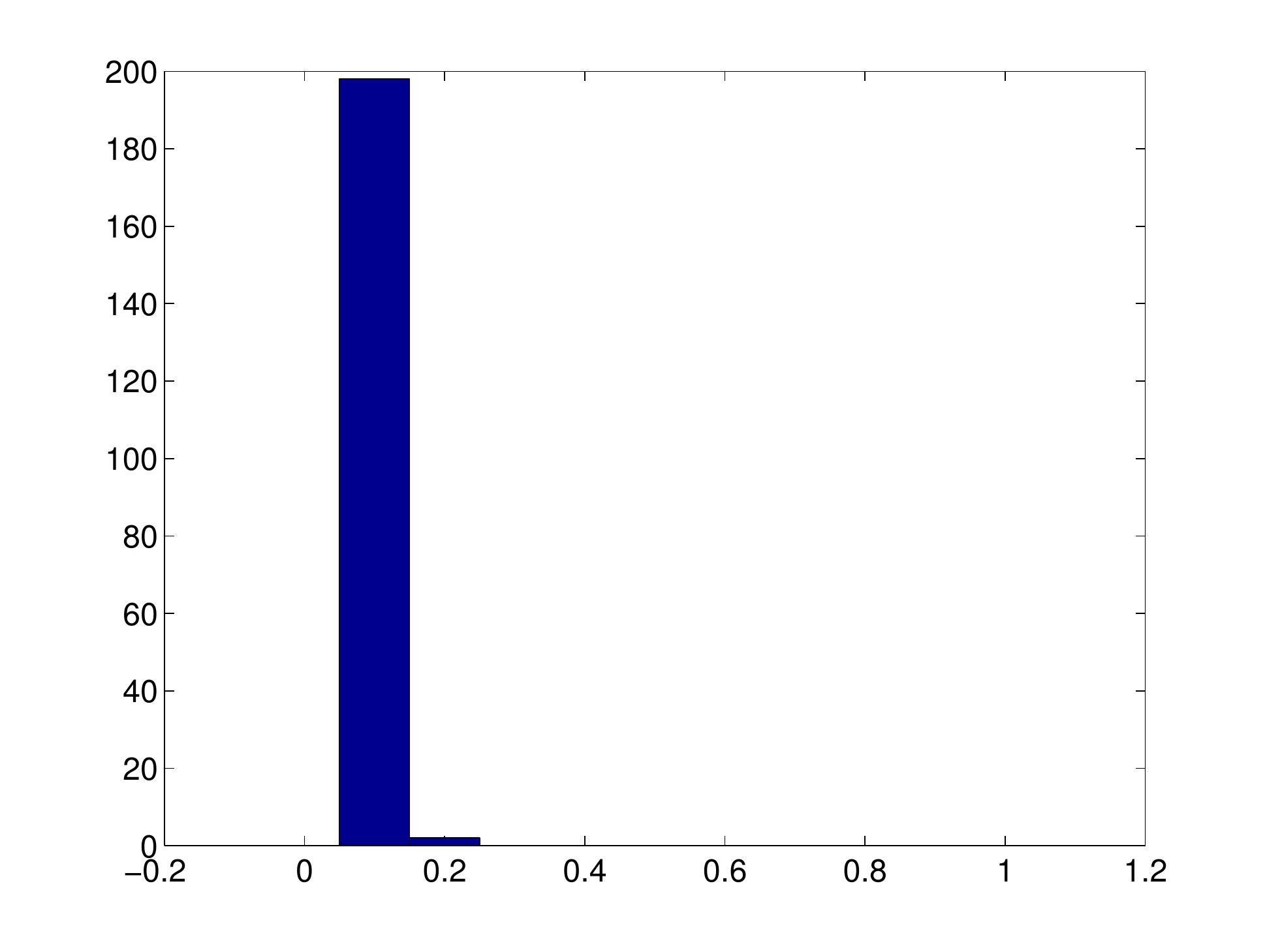}
\includegraphics[width=0.32\textwidth]{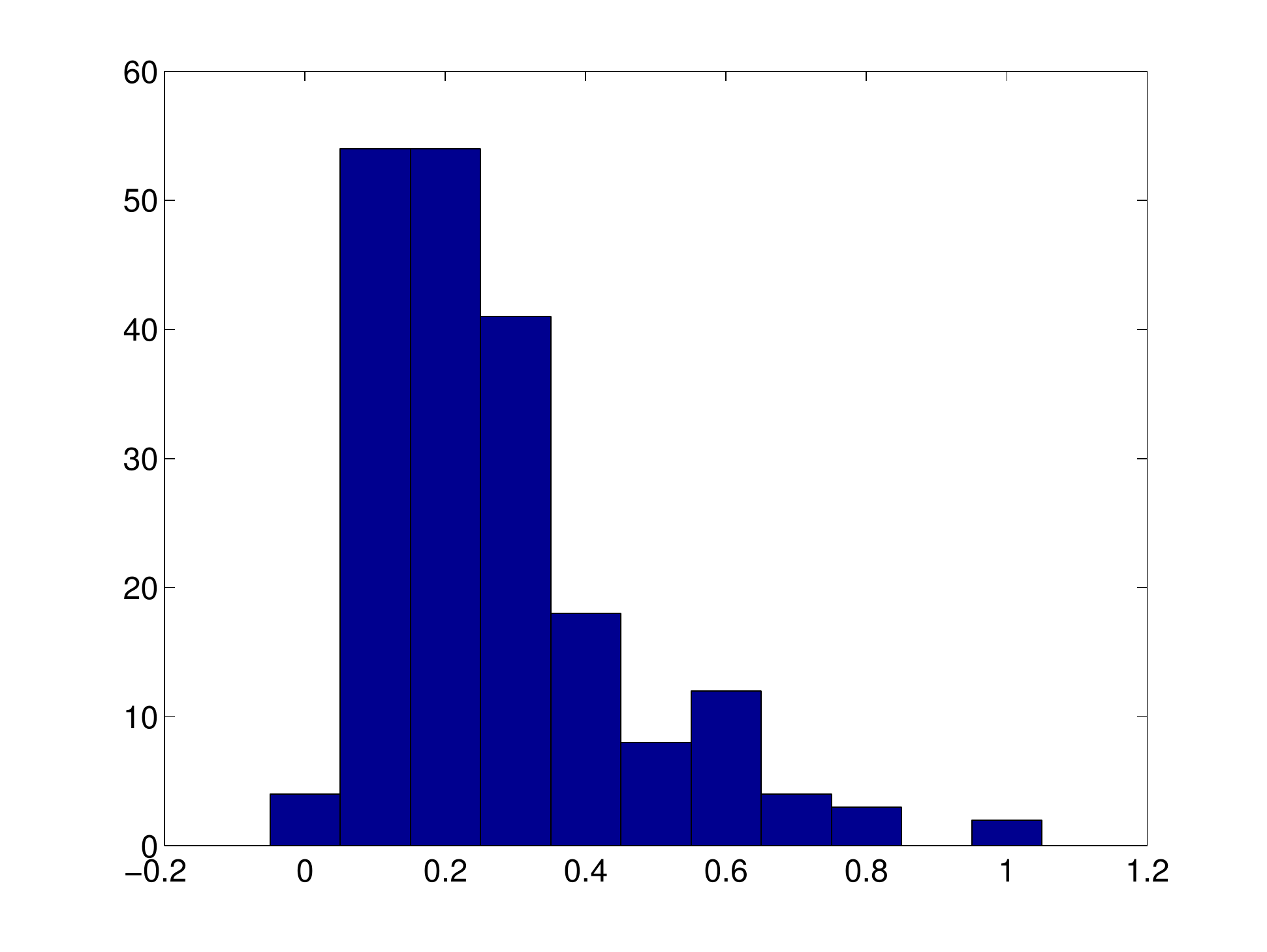}
\includegraphics[width=0.32\textwidth]{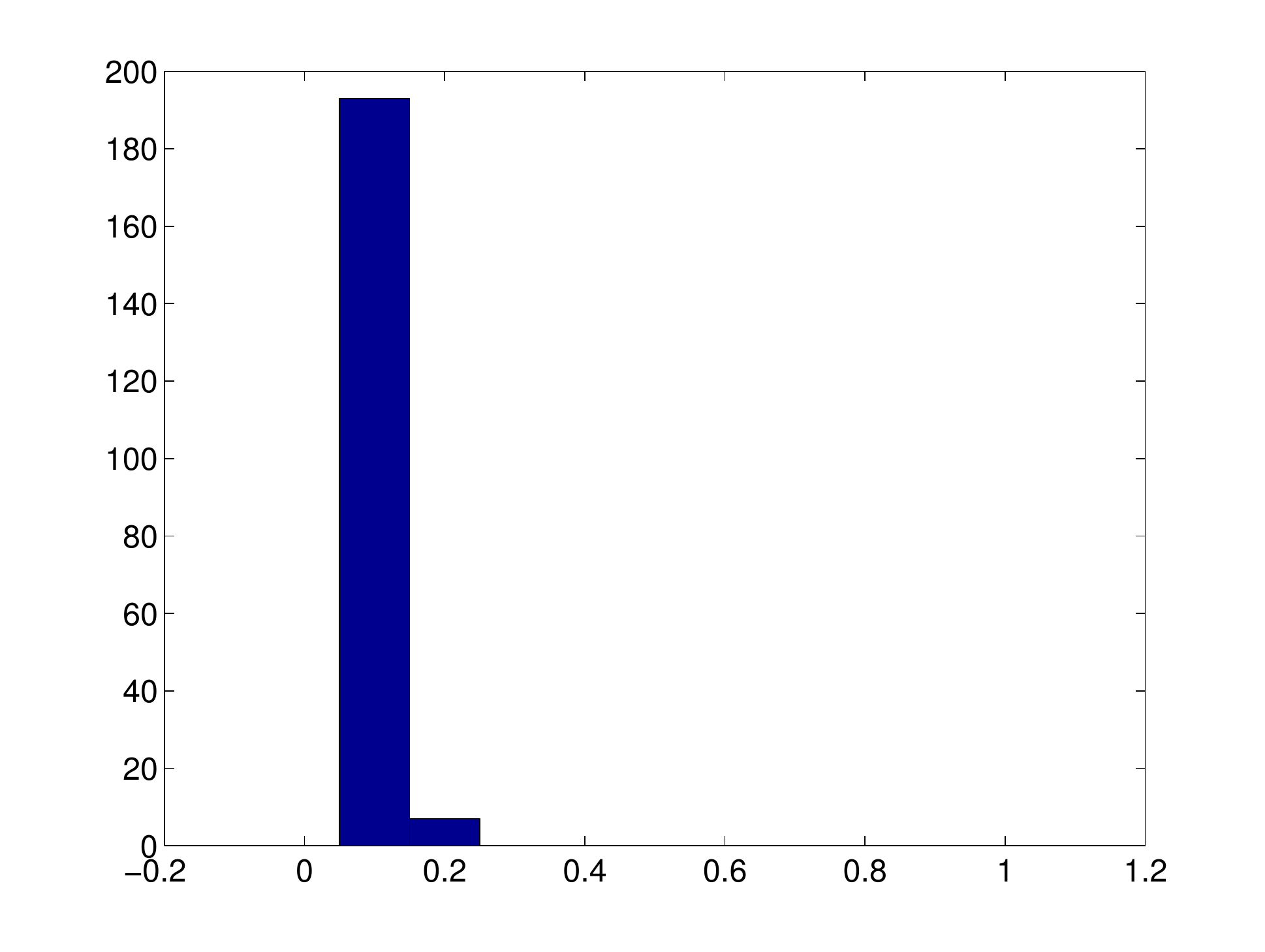}
\includegraphics[width=0.32\textwidth]{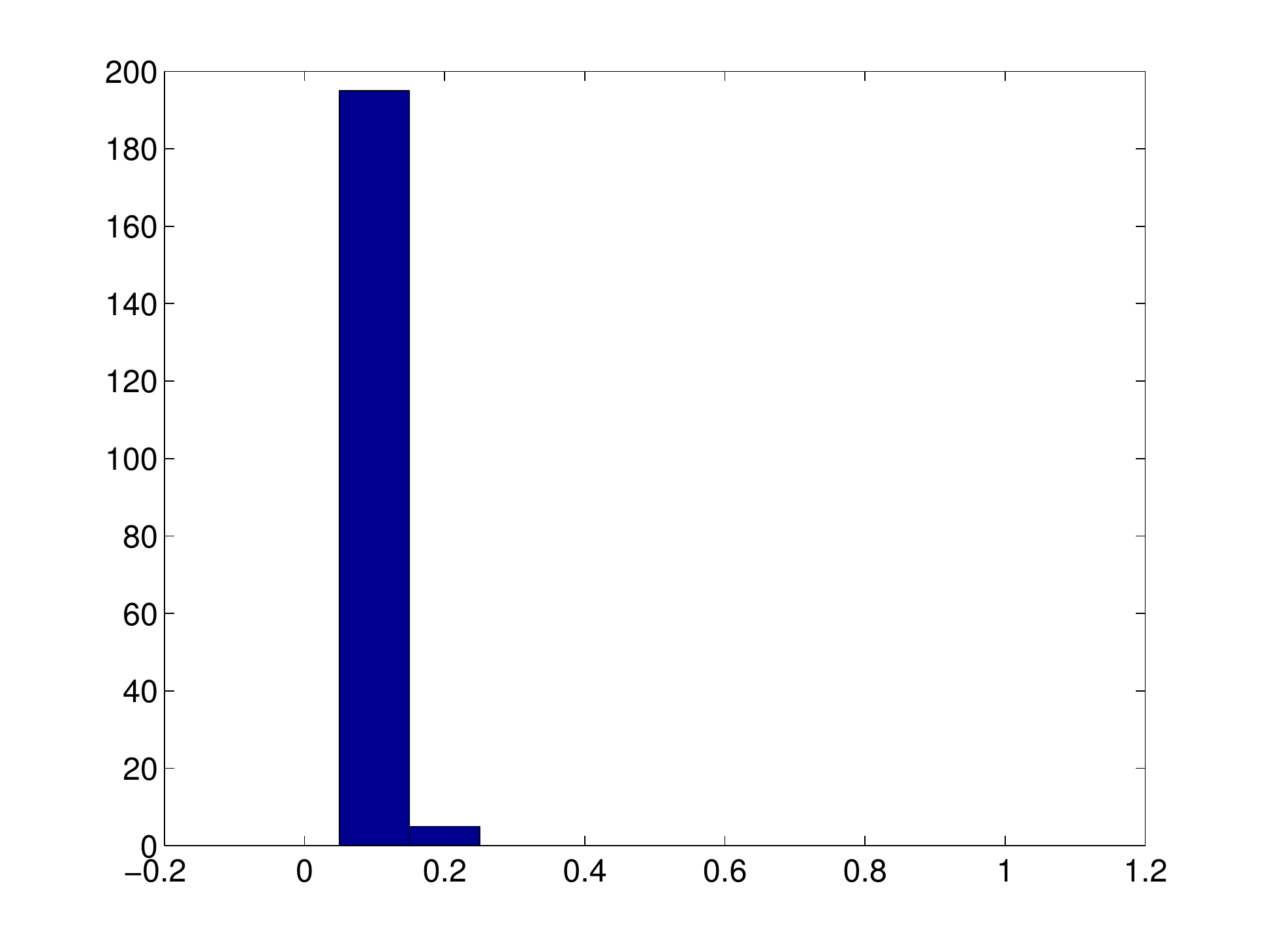}

\caption{{  Figure shows the histogram  $\| x_0x_0^\top-x^*{x^*}^\top\|$ of 200 trials under the noise effect.  Left, middle and right columns show the results with  rank  $r=1$, $r=2$ and  $r=3$.  Here, we use the proposed initialization in the top row and the random initialization in the second row. }}
\label{Noise1}
\end{center}
\end{figure}

\subsection{Phase retrieval experiments}
Next, we report phase retrieval simulation results (Fourier matrices), with    $x_0$ being real, positive images. Images are reconstructed subject to the positivity constraints ( i.e., the leading singular vector).  The results are provided to show some advantage of  ADM with $r=2$  over ADM with $r=1$. Here we use $\beta=0.1$ in the following experiment.

According to our experience, the phase retrieval with the Fourier matrix is a very difficult problem, in particular in the presence  of noise.   
To alleviate the difficulty, researchers have suggested random illumination  to enforce the uniqueness of solutions~\cite{IOPORT.06071995}.
It is known that the phase retrieval has a unique solution up to  three classes: constant global phase,   spatial shift, and conjugate inversion. With high probability absolute uniqueness holds with a random phase illumination; see Cor. 1~\cite{IOPORT.06071995}.Our experiences show that the random phase illumination works much better than the above uniform illumination. 

In Fig.~\ref{FFTlena}, we demonstrate the the ADM with $r=1,2$ on the  images with random phase illumination.  Let $x_0\in \mathbf{R}^{300\times 300}$ be the intensity of the Lena image\footnote{We downsample the Lena image from http://www.ece.rice.edu/~wakin/images/ by approximately  a factor $2$ and use  zero padding with the oversampling rate\cite{Millane:90}\cite{Miao:98} $1.23$. }, see the bottom subfigure. We add noise and generate the data
\[
b^2=\max(|Ax_0|^2+noise,0),
\]
where $A$ is the Fourier matrix. The SNR is $39.8dB$ and the oversampling is $1.23$. Reconstruction errors 
$\|\frac{x^*}{\|x^*\|_F}-\frac{x_0}{\|x^*_0\|_F}\|_F$ for rank one and rank two are $0.126$ and $0.109$, respectively. The ADM with $r=2$ has a better reconstruction.
 
\begin{figure}[htbp]
\begin{center}

\includegraphics[width=0.4\textwidth]{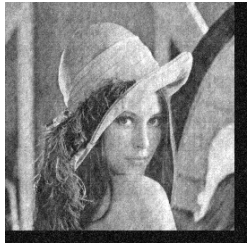}
\includegraphics[width=0.4\textwidth]{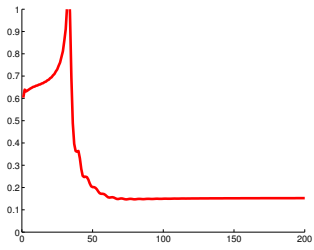}
\includegraphics[width=0.4\textwidth]{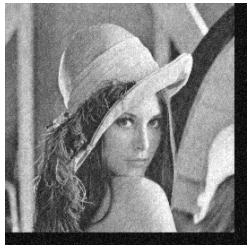}
\includegraphics[width=0.4\textwidth]{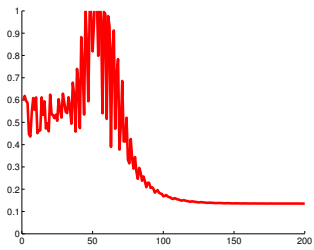}
\includegraphics[width=0.4\textwidth]{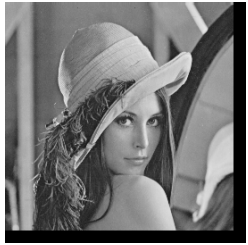}

\caption{\bf Figures show the reconstructed images and the error $\| |AA^\dagger z|-b\|_F$ vs. the number of iteration via ADM with rank  $r=1$ (the first row) and $r=2$ (the second row). 
 }

\label{FFTlena}
\end{center}
\end{figure}

\subsection{Conclusions}
In this paper, we discuss the rank-one matrix recovery via two   approaches. First, the rank-one matrix is computed among the Hermitian matrices as in PhaseLift. We make the observation that  matrices in the feasible set have  equal trace norm via the measurement matrices with orthonormal columns. Experiments show that with the aid of these orthogonal frames, exact recovery occurs under a smaller $N/n$ ratio compared with the  PhaseLift in both real and complex cases.
In the second part of the paper, we discuss the ``lifting'' of  the nonconvex  alternating  direction minimization method from rank-one to rank-$r$ matrices, $r>1$. The benefit of this relaxation  cannot be overestimated, because the construction of large Hermitian matrices is avoided, as is the associated Hermitian matrices projection.  Comparing with the ADM with rank-one, the ADM with rank $r>1$ performs better  in recovering noise-contaminated signals, which is demonstrated in simulation experiments.

Another contribution is  the error estimate between  the unknown signal and the singular vector corresponding to the least singular value.
 The initialization has an effect of importance in the nonconvex minimization. We demonstrate that a good initialization can be the least singular vector of the subset of sensing vectors corresponding to the small measurement values $b_i$. In the case of real Gaussian matrices, the error  can be reduced, as the number of measurements grows at a rate proportional to the dimension of  unknown signals. 
One of our  future works is the generalization of the error estimate to complex frames, in particular  the case of  
 the  Fourier matrix.

\appendix
\section{Standardization of $A$}

In the following, we will prove Theorem~\ref{Standard} in several steps. 
We discuss the existence first. The uniqueness analysis will be shown later.  Fixing $A$, let $\mathcal{D}$ be  the  inverse matrices of diagonal matrices $D$, \[ \mathcal{D}:=\{D^{-1}\in \mathbf{R}^{N\times N}: \|D^{-1/2} A\|_F^2=\sum_{i=1}^N D_{i,i}^{-1} \sum_{j=1}^n A_{i,j}^2=N, D_{i.i}\ge 0\}.\]
Clearly $\mathcal{D}$ is nonempty and convex compact. 
In fact, $D_{i,i}$ has a positive lower bound, \[ D_{i,i}\ge N^{-1} \sum_{j=1}^n A_{i,j}^2 \textrm{ for all $i$.}\] For each $D^{-1}\in \mathcal{D}$, let  $f: \mathcal{D}\to \mathcal{D}$  be the   function    \[  f(D^{-1})=\hat D^{-1}, \textrm{  where
$QB=D^{-1/2}A$ is the QR factorization},\]
 and  each row of 
$\hat D^{-1/2}A B^{-1}$ has   norm one.  In fact, the function $f$ generates iterations  $\{(D^{k})^{-1}\}_{k=0}^\infty$ with  
 $(D^{k+1})^{-1}=f((D^{k})^{-1})$.
That is, start with $Q^0=A$. Repeat the two steps for $k=0,1,2,\ldots$ until it converges: 
\begin{eqnarray*}
&(ii)& \textrm{ Normalize the row of $Q^k$ by $(D^k)^{-1/2}Q^k$;} \\
&(ii)& \textrm{Take the QR factorization:}
 (D^k)^{-1/2}Q^{k-1}=Q^{k}R^{k}.
\end{eqnarray*}
Since $D^{-1/2} A$ has rank $n$,   then $B$ has rank $n$ and $B^{-1}$ exists. The function $f$ is well defined: Once $B$ is given, then  choose the diagonal matrix $ D$ to be that which  normalizes the rows of  $A B^{-1}$.
According to  
Brouwer's fixed-point theorem, we have the existence of $D$, such that $D^{-1/2}AB^{-1}=Q$ consists of orthogonal columns and each row has  norm one.

Before the uniqueness proof, we  state  one equation of $D$.
\begin{prop}
The diagonal matrix $D$ satisfies the equation,
\begin{equation}\label{EqD} (n/N) D_{i,i}=(A (A^\top D^{-1} A)^{-1}A^\top )_{i,i}.\end{equation}
\end{prop}
\begin{proof}
 
 According to $D^{-1/2} A=QB $, we have
 \[
 (n/N)D_{i,i}=(D^{1/2}QQ^\top D^{1/2})_{i,i}=(A (B^\top B)^{-1} A^\top )_{i,i}.
 \]
 Note that  $B^\top B=A^\top D^{-1/2} Q^\top Q D^{-1/2} A=A^\top D^{-1} A$. Thus, \[
 \; (n/N)D_{i,i}=(A (A^\top D^{-1} A)^{-1}A^\top )_{i,i}.
\]
\end{proof}

\begin{prop}\label{Jens}
Let $p_i\in (0,1)$, $i=1,\ldots, n$ with $\sum_{i=1}^n p_i=1$. Let $\lambda_i>0$ for $i=1,\ldots, n$. Then 
\[
\sum_{i=1}^n p_i\lambda_i \ge (\sum_{i=1}^n p_i\lambda_i^{-1})^{-1},
\]
where equality holds if and only if $\{p_i\}_{i=1}^n$ are equal.
\end{prop}
\begin{proof}
Let $f(x)=x^{-1}$ for $x>0$, which is strictly  convex. The statement is the application of Jensen inequality,
\[\sum_{i=1}^n p_i\lambda_i^{-1}\ge 
(\sum_{i=1}^n p_i\lambda_i )^{-1}.
\]

\end{proof}

\begin{prop}
Suppose that $Q$ is a standardized matrix satisfying the rank* condition.
Let $F^1$ be a positive diagonal matrix. Then   the iteration
\[
F^{k+1}_{i,i} =(N/n) (Q(Q^\top (F^k)^{-1} Q )^{-1}Q^\top )_{i,i},\; k=1,\ldots.
\]
yields $\lim_{k\to \infty } F^k=cI_{N\times N}$, where $c$ is some scalar.
\end{prop}

\begin{proof}
We will show $ tr(F^{k+1})\le  tr(F^{k})$. Suppose that 
$\{(\lambda_i^{-1},q_i)\}_{i}$ are  eigenvalues-eigenvectors  of $Q^\top (F^k)^{-1}Q$, 
then $\{(\lambda_i,q_i)\}_{i}$ are eigenvalues-eigenvectors of   $(Q^\top (F^k)^{-1}Q)^{-1}$. Hence, 
\[
\lambda_i^{-1}=q_i^\top  Q^\top (F^k)^{-1}Qq_i, 
\]
and
 \[  tr(F^{k+1})=\sum_{j=1}^N \mu^{k+1}_j=(N/n)\sum_{i=1}^n \lambda_i= (N/n)\sum_{i=1}^n (\sum_{j=1}^N(F^k_{j,j})^{-1}(Qq_i)_{j}^2)^{-1} .\] 
 
Denote the j-th entry of  $|(Qq_i)_{j}|$ by $p_{j,i}$. Then $\sum_{j=1}^N p_{j,i}^2=1$ and $\sum_{i=1}^n p_{j,i}^2=n/N$. Let $\{ \mu_i^{k}\}_{i=1}^N$ be the diagonal entries of $F^{k}$. 
Then 
\[
\sum_{j=1}^N \mu_j^{k+1}=(N/n)\sum_{i=1}^n(\sum_{j=1}^N  (\mu^k_{j})^{-1} p_{j,i}^2)^{-1}\le (N/n)\sum_{i=1}^n\sum_{j=1}^N  \mu^k_{j} p_{j,i}^2=\sum_{j=1}^N\mu^k_j,
\]
where the last equality is due to $\sum_{i=1}^n p_{j,i}^2=n/N$.
Hence, $ tr(F^{k+1})\le  tr(F^k)$. Denote one of limiting points of $\mu^{k}_i$ by $\mu^*_i$ and then 
\[
(\sum_{j=1}^N p_{j,i}^2(\mu^*_j)^{-1})^{-1}=\sum_{j=1}^N p_{j,i}^2 \mu^*_j \textrm{ for all } i. 
\]
Hence, $\mu^*_i=\mu^*_j$ for all $ i,j$ with $p_{j,i}> 0$. Due to the rank* condition, $QV$ cannot be written in the form of Eq.~(\ref{Block}) for any  orthogonal matrix  $V$ whose columns are orthonormal vectors $\{q_i\}_{i=1}^n$ with $p_{j,i}=|(QV)_{j,i}|$.
Hence, $c=\mu^*_i=\mu^*_j$
 for all $i,j$.

\end{proof}

Finally, we complete the proof in the following.
\begin{prop}
Suppose that  $A$ satisfies the rank* condition. 
 Let $D_*$ be one solution of Eq.~(\ref{EqD}). Then with any positive diagonal matrix $D^0$, 
the iteration 
\[
D^{k+1}=(N/n)diag(A(A^\top (D^k)^{-1} A)^{-1} A^\top ),\; k=1,\ldots, 
\]
yields 
\[
\lim_{k\to \infty} D^k=D_*.
\]
Thus, $D_*$ is unique.
\end{prop}
\begin{proof} 
Let   $D_*^{-1/2}A=Q B$ be the QR factorization of $D_*^{-1/2}A$. Then 
\[(A^\top D^{-1}A)^{-1}= B^{-1} (Q^\top (D_*^{-1/2} D D_*^{-1/2})^{-1} Q)^{-1}B^{-1},\]
and the iteration becomes
\[
D^{-1/2}_*D^{k+1}D^{-1/2}_* =(N/n)diag(Q(Q^\top ( D_*^{-1/2}D^k D_*^{-1/2})^{-1} Q)^{-1} Q^\top ),\; k=1,\ldots.
\]
Let $F^{k}=D^{-1/2}_*D^{k}D^{-1/2}_*$.
Since  $A$  satisfies the rank* condition, then for any nonsingular matrix $B$, $D_*^{-1/2}AB^{-1}$ also satisfies the rank* condition and
cannot be written in the form in
 Eq.~(\ref{Block}) for any  orthogonal matrix.
According to Prop.~\ref{Jens},  the proof is completed.

\end{proof}
%
%
%


\bibliographystyle{plain}	
\bibliography{ShortNote}		

\end{document}